\documentclass[12pt]{article}
\usepackage{fancyhdr, array,calc,graphicx,url,tabularx}
\usepackage{geometry}
\usepackage{latexsym}
\usepackage{amssymb}
\usepackage{amsmath}
\usepackage{makeidx}
\usepackage{pdfsync}
\usepackage{undertilde}

\makeindex

\renewcommand{\gg}{\gamma}

\newcommand{\rest}{\restriction}
\newcommand{\la}{\langle}
\newcommand{\ra}{\rangle}

\newcommand{\card}[1]{{\vert #1 \vert} }

\renewcommand{\models}{\vDash}
\newcommand{\powerset}{{\cal P}}
\newcommand{\cp}{{\rm cp }}

\newcommand{\cf}{{\rm cf}}

\newtheorem{theorem}{Theorem}[section]
\newtheorem{proposition}[theorem]{Proposition}
\newtheorem{definition}[theorem]{Definition}
\newtheorem{remark}[theorem]{Remark}

\newtheorem{lemma}[theorem]{Lemma}
\newtheorem{corollary}[theorem]{Corollary}

\newtheorem{conjecture}[theorem]{Conjecture}
\newtheorem{question}[theorem]{Question}

\newtheorem{problem}{Problem}

\numberwithin{figure}{section}

\newenvironment{proof}{{\it{
Proof.}}}{\nopagebreak\mbox{}{\hfill$\square$}
\par\smallskip}
\newenvironment{Claim}{\noindent{\it
Claim} {}}{}
\newenvironment{Case}{\noindent{\it
Case}}{}

\newcommand{\rcon}[1]{Conjecture~\ref{#1}}

\newcommand{\rthm}[1]{Theorem~\ref{#1}}
\newcommand{\rlem}[1]{Lemma~\ref{#1}}
\newcommand{\rprob}[1]{Problem~\ref{#1}}
\newcommand{\rcor}[1]{Corollary~\ref{#1}}
\newcommand{\rdef}[1]{Definition~\ref{#1}}
\newcommand{\rfig}[1]{Figure~\ref{#1}}

\newcommand{\rsec}[1]{Section~\ref{#1}}

\newcommand{\rrem}[1]{Remark~\ref{#1}}

\def\inseg{\trianglelefteq}

\def\k{\kappa}
\def\a{\alpha}
\def\b{\beta}
\def\d{\delta}

\def\l{\lambda}

\def\P{{\mathcal{P} }}

\def\Q{{\mathcal{ Q}}}

\def\R{{\mathcal R}}

\def\H{{\rm{HOD}}}
\def\M{{\mathcal{M}}}
\def\N{{\mathcal{N}}}

\def\T {{\mathcal{T}}}
\def\U{{\mathcal{U}}}
\def\S{{\mathcal{S}}}

\def\VT{{\vec{\mathcal{T}}}}
\def\VU{{\vec{\mathcal{U}}}}

\def\cp #1{{ crit  #1 }}
\def\card#1{\left|#1\right|}

\def\iff{\mathrel{\leftrightarrow}}

\def\and{\mathrel{\kern1pt\&\kern1pt}}

\def\<#1>{\langle\,#1\,\rangle}

\input xy
\xyoption{all}

\title{Descriptive inner model theory 
}
\author{Grigor Sargsyan \thanks{This material is partially based upon work supported by the National Science Foundation under Grant No DMS-0902628. Part of this paper was written while the author was a Leibniz Fellow at the Mathematisches Forschungsinstitut Oberwolfach.}\\
        Department of Mathematics\\
        Rutgers University\\
        Hill Center for the Mathematical Sciences\\
        110 Frelinghuysen Rd.\\
        Piscataway, NJ 08854 USA\\
        http://math.rutgers.edu/$\sim$gs481\\
        grigor@math.rutgers.edu}
\date{Dedicated to my father Edvard Sargsyan\\ on the occasion
     of his $60$th birthday\\ \today}

\begin{document}

\maketitle

\baselineskip=18pt

\begin{abstract}
The purpose of this paper is to outline some recent progress in descriptive inner model theory, a branch of set theory which studies descriptive set theoretic and inner model theoretic objects using tools from both areas. There are several interlaced problems that lie on the border of these two areas of set theory, but one that has been rather central for almost two decades is the conjecture known as the \textit{Mouse Set Conjecture} (MSC). One particular motivation for resolving MSC is that it provides grounds for solving the \textit{inner model problem} which dates back to 1960s. There have been some new partial results on MSC and the methods used to prove the new instances suggest a general program for solving the full conjecture. It is then our goal to communicate the ideas of this program to the community at large.
\end{abstract}

The program of constructing canonical inner models for large cardinals has been a source of increasingly sophisticated ideas leading to a beautiful theory of canonical models of fragments of set theory known as \textit{mice}. It reads as follows:\\

\textbf{The inner model program.} Construct canonical inner models with large cardinals.\\

It is preferred that the constructions producing such inner models are applicable in many different situations and are independent of the structure of the universe. Such universal constructions are expected to produce models that are ``canonical": for instance, the set of reals of such models must be definable. Also, it is expected that when the universe itself is complicated, for instance, \textit{Proper Forcing Axiom (PFA)} holds or that it has large cardinals, then the inner models  produced via such universal constructions have large cardinals. To test the universality of the constructions, then, we apply them in various situations such as under  large cardinals, PFA  or under other combinatorial statements known to have large cardinal strength, and see if the resulting models indeed have significant large cardinals. Determining the consistency strength of PFA, which is the forward direction of the \textit{PFA Conjecture} stated below, has been one of the main applications of the inner model theory and many partial results have been obtained (for instance, see \rthm{steel pfa} and \rthm{grigor pfa}). The reverse direction of the PFA Conjecture is a well-known theorem due to Baumgartner. 

\begin{conjecture}[The PFA Conjecture]\label{pfa conjecture} The following theories are equiconsistent.
\begin{enumerate}
\item ZFC+PFA.
\item ZFC+``there is a supercompact cardinal".
\end{enumerate}
\end{conjecture}

Recently, in \cite{VW}, Viale and Weiss showed that if one forces PFA via a proper forcing then one needs to start with a supercompact cardinal. This result is a strong evidence that the conjecture must be true. 

One approach to the inner model program has been via descriptive set theory. \rsec{partial results on imp} contains some details of such an approach. \rrem{dimt approach to imp} summarizes this approach in more technical terms then what follows. The idea is to use the canonical structure of models of determinacy to squeeze large cardinal strength out of stronger and stronger theories from the \textit{Solovay hierarchy} (see \rsec{solovay hierarchy}). This is done by translating the descriptive set theoretic strength into the language of inner model theory. Once such a translation is complete, we get canonical inner models for large cardinals by examining the inner model theoretic objects that the translation procedure produces. 

Descriptive inner model theory is the theory behind the approach to the inner model problem described above. Its main technical goal is to construct canonical models of fragments of set theory, namely mice, that capture the truth holding in the universe. An example of such capturing is the \textit{Shoenfield's absoluteness theorem:} $\Sigma^1_2$ statement is true iff it is true in $L$, the constructible universe. There is a natural way of organizing the construction of such mice as an induction. In this induction, we aim to construct canonical mice that capture complicated \textit{universally Baire} sets of reals, which are sets of reals whose continuous preimages in any compact Hausdorf space have the property of Baire.  Under various assumptions such as large cardinal assumptions or determinacy assumptions, the continuous reducibility, or \textit{Wadge reducibility}, restricted to universally Baire sets is a well-founded relation, called \textit{Wadge order}, in which each universally Baire set has the same rank as its complement.  The ranking function of Wadge order gives a stratification of universally Baire sets into a hierarchy of sets whose $\a$th level consists of all universally Baire sets of rank $\a$. The main technical goal of descriptive inner model theory is then to inductively construct mice capturing the universally Baire sets of each rank. The first step of the induction is the Shoenfield's absoluteness theorem. 
The next steps then are the constructions of canonical models that are correct about $\Sigma^1_3$, $\Sigma^1_4$, $\Sigma^1_5$,..., $(\Sigma^2_1)^{L(\mathbb{R})}$ and etc. 

The methods of descriptive inner model can be used to attack the PFA Conjecture and similar problems via a method known as the  \textit{core model induction} (see \rsec{cmi}). Recent results, using the core model induction, imply that the descriptive set theoretic approach to the inner model program has the potential of settling the forward direction of the PFA Conjecture (see \rsec{cmi}). 

At the heart of descriptive inner model theoretic approach to the inner model program lies a conjecture known as the \textit{Mouse Set Conjecture} (MSC) (see \rsec{intro to msc}), which essentially conjectures that the most complicated form of definability can be captured by canonical models of set theory, namely mice. The author, in \cite{ATHM}, proved instances of MSC (see 2 of the Main Theorem), and one of the goals of this paper is to give a succinct synopsis of \cite{ATHM}. In particular, we will concentrate on the theory of \textit{hod mice} and explain how it was used to prove instances of MSC. The paper, however, is written for non experts and because of this we have included some topics that lead to MSC and to a far more general topic, the analysis of $\H$ of models of determinacy. However, our goal wasn't to write a historical introduction to the subject. We apologize for our modest selection of key ideas that shaped the subject in the last 50 years. 

To mention a few other applications of hod mice we make the following two definitions. An ideal is called \textit{precipitous} if the generic ultrapower associated with it is wellfounded. A precipitous ideal $I$ is called \textit{super homogeneous} if it is homogeneous (as a poset) and the restriction of the generic embedding on ordinals is independent of the generic object. $AD^+$ is an extension of $AD$ due to Woodin. Readers unfamiliar with $AD^+$ can take it just to mean $AD$ without losing much, for now (see \rsec{solovay hierarchy}). A pair of transitive models $M$ and $N$ are called \textit{divergent models} of $AD^+$ if $Ord\subseteq M\cap N$, $\mathbb{R}\in M\cap N$,  $\powerset(\mathbb{R})^M\not \subseteq N$ and  $\powerset(\mathbb{R})^N\not \subseteq M$. 

The following theorem is the summary of the results that will be the focus of the rest of this paper. Recall that
\begin{center}
$\Theta=_{def}\sup\{ \a :$ there is a surjection $f: \mathbb{R}\rightarrow \a\}$.
\end{center}
Also, recall that $AD_{\mathbb{R}}$ is the axiom asserting that all two player games of perfect information on reals are determined. 
We let $AD_{\mathbb{R}}+``\Theta$ is regular" stand for the theory $ZF+AD_{\mathbb{R}}+``\Theta$ is a regular cardinal".  \rsec{solovay hierarchy} has more on $AD_{\mathbb{R}}+``\Theta$ is regular".

\begin{theorem}[Main Theorem, \cite{ATHM}]\label{the Main Theorem} Each of the following 4 statements implies that there is an inner model containing the reals and satisfying $AD_{\mathbb{R}}+``\Theta$ is regular".
\begin{enumerate}
\item $CH+``$for some stationary $S\subseteq \omega_1$ the non-stationary ideal restricted to $S$ is $\omega_1$-dense and super homogeneous".
\item \textit{Mouse Capturing} (see \rsec{intro to msc}) fails in some inner model $M$ containing the reals and satisfying $AD^++``V=L(\powerset(\mathbb{R}))"$.
\item There are divergent models of $AD^+$.
\item There is a Woodin cardinal which is a limit of Woodin cardinals (see \rsec{extenders}).
\end{enumerate}
\end{theorem}

All parts of the Main Theorem were proved using the theory of hod mice and most of the proof can be found in \cite{ATHM}. The proof from 1, however, is unpublished but can be found in \cite{thesis}.

Woodin, in unpublished work, showed that the hypothesis of 1 is consistent relative to $AD_{\mathbb{R}}+``\Theta$ is regular". Combining this result with the implication from 1 of the the Main Theorem, we get the following.

\begin{theorem}[S.-Woodin]\label{sw} The following theories are equiconsistent.
\begin{enumerate}
\item  $AD_{\mathbb{R}}+ ``\Theta$ is regular".
\item $CH+``$for some stationary $S\subseteq \omega_1$ the non-stationary ideal restricted to $S$ is $\omega_1$-dense and super homogeneous".
\end{enumerate}
\end{theorem}

We will explain some more applications of the theory of hod mice in \rsec{applications} and in particular, will outline the proof of 3 of the Main Theorem, which has an interesting consequence. Below $MM(c)$ stands for \textit{Martin's Maximum} for partial orders of size continuum. Part 3 of the Main Theorem can be used to get an upper bound for $MM(c)$.

\begin{corollary}\label{upper bound for mm(c)} Con(ZFC+``There is a Woodin limit of Woodin cardinals") implies $Con(ZFC+MM(c))$.
\end{corollary}
\begin{proof}
In \cite{Woodin}, Woodin showed that one can force a model of $MM(c)$ over a model of $AD_{\mathbb{R}}+``\Theta$ is regular". Also, Woodin showed that the existence of divergent models of $AD^+$ is at most as strong as a Woodin limit of Woodin cardinals (see \cite{EA}). Thus, it follows from 3 of the Main Theorem that $MM(c)$ is weaker than a Woodin limit of Woodin cardinals.
\end{proof}

Notice that Woodin's result mentioned in the proof of \rcor{upper bound for mm(c)} is a consistency result and cannot be used to prove the conclusion of the the Main Theorem from part 4 using part 3. It was believed that both the hypothesis in part 2 of \rthm{sw} and $MM(c)$ have a very strong large cardinal strength, at least as strong as a supercompact cardinal. Moreover, both of these statements can be forced from large cardinals in the region of supercompact cardinals. It was then natural to conjecture, as Woodin did, that $AD_{\mathbb{R}}+``\Theta$ is regular" is very strong as well. However, 4 of the Main Theorem implies that the theory
$AD_{\mathbb{R}}+``\Theta$ is regular" is quite weak and \rthm{theta regular from large cardinals} implies that it is much weaker than a Woodin limit of Woodin cardinals. It is known, however, that it is stronger than a proper class of Woodin cardinals and strong cardinals (see \cite{DMATM} and \cite{DMT}). Recently the author and Yizheng Zhu have determined the exact large cardinal strength of $AD_{\mathbb{R}}+``\Theta$ is regular". The definition of this cardinal appears in \rsec{consistency of sphi}.    

In what follows, we will give a more detailed explanations of the concepts mentioned thus far and in particular, will outline the proofs of part 2 and 3 of the Main Theorem. The paper is organized as follows. We assume familiarity with introductory chapters of \cite{Jech}. Sections 1 is an introduction to the inner model problem. Here is where the bulk of inner model theoretic notions are introduced. This section also introduces several fundamental notions such as \textit{extenders}. Section 2 introduces the Solovay hierarchy and this is where the bulk of descriptive set theoretic notions are introduced.  \rsec{intro to msc} is where MSC is stated and there is an outline of its proof below $AD_{\mathbb{R}}+``\Theta$ is regular" in \rsec{proof of msc}. \rsec{hod is a hod premouse} is devoted to the analysis of $\H$ of models of $AD^+$ and in particular, to the representation of $\H$ as a hod premouse. The readers can find an extensive discussion of hod mice in \rsec{hod mice}.  \rsec{applications} contains the proof of part 3 of the Main Theorem.


 The main source of ideas leading to this paper have been my private conversations with John Steel during the years I spent in Berkeley and the two conferences on Core Model Induction and Hod Mice that took place in Muenster, Germany during the Summers of 2010 and 2011. I am indebted to John Steel for those conversations and I would like to thank the organizers and participants of these conferences for many wonderful conversations on the subject. I would like to express my gratitude to John Baldwin, Ilijas Farah, Paul Larson, Ralf Schindler, John Steel and Trevor Wilson for many constructive comments on the earlier versions of this paper. I would especially like to thank Ilijas Farah. Without his encouragements this paper just wouldn't have materialized. I am also indebted to the referee for the very long list of fundamental improvements. 
 
 \section{The inner model problem}\label{inner model theory}

The origin of the inner model program goes back to the 1960s. Soon after G\rm{\"{o}}del's introduction of the constructible universe $L$, Scott showed that the existence of a measurable cardinal implies that $V\not =L$. The constructible universe has an astonishingly canonical structure immune to forcing, and this is one of the reasons Scott's result is so intriguing. It is impossible not to ask whether large cardinals can have models resembling $L$. Or, perhaps large cardinals just outright imply the existence of certain \textit{non-canonical} structures. If so, then what could these structures be? 

Since 60s, the vagueness of ``resembling $L$" has been gradually removed. The first steps were taken by Kunen, Silver and Solovay who considered the model $L[\mu]$ where $\mu$ is a non-principal normal\footnote{i.e., whenever $f:\kappa\rightarrow \kappa$ is such that $\{ \a: f(\a)<\a\}\in \mu$ there is some $\b<\k$ such that $\{ \a: f(\a)=\b\}\in\mu$.} $\kappa$-complete ultrafilter over some $\kappa$. Kunen and Silver showed that $L[\mu]$ has many $L$-like properties. For instance, Silver showed that $L[\mu]\models GCH$ and Kunen showed that $L[\mu]$ is unique in the following sense: if $\mu$ and $\nu$ are two non-principal normal $\kappa$-complete ultrafilters over $\k$ then $L[\mu]=L[\nu]$.

The next wave was the introduction of larger models called \textit{extender models} by Mitchell (see \cite{Mitchell74} and \cite{Mitchell83}). An \textit{extender}, defined more precisely in the next subsection, is a family of ultrafilters having a certain coherence property. Like ultrafilters, extenders code an elementary embedding of the universe into some inner model.  Extender models have the form $L_\a[\vec{E}]$ where $\a$ is an ordinal and $\vec{E}$ is a \textit{coherent sequence of extenders}, which must have several properties in order for the resulting model to be $L$-like or canonical. Mitchell's models could carry many strong cardinals but they couldn't have Woodin cardinals in them. Later, in \cite{IT}, Martin and Steel defined a Mitchell style extender model that can carry cardinals as large as superstrong cardinals.

However, the internal theory of Martin-Steel models is difficult to determine and in fact, it is not even known if $GCH$ holds in all Martin-Steel models. \textit{Fine structural analysis} of extender models, pioneered by Jensen in \cite{Jensen}, was a crucial step towards defining truly $L$-like models, models that have the same combinatorial structure as $L$. Such fine structural extender models are called \textit{mice} (the terminology is due to Jensen). In \cite{FSIT}, Mitchell and Steel defined the modern notion of mice and developed their basic fine structure.

Equipped with the notion of a mouse, the inner model program can then be made more precise as follows. \\

\textbf{The inner model problem.} Construct mice satisfying large cardinal hypothesis as strong as possible.\\

 While the problem asks for an actual construction of a mouse, even defining and developing a reasonable notion of a mouse in such a way that it could carry large cardinals at the level of a supercompact cardinal is a very difficult problem. Mitchell-Steel mice can only carry large cardinals in the region of superstrong cardinals, which are much weaker than supercompact cardinals. In \cite{Neeman}, Neeman constructed mice with many Woodin limits of Woodin cardinals, which are much weaker than superstrong cardinals. Neeman's result remains the best partial result on the inner model problem. However, Neemsn's construction isn't flexible enough for applications in calibrations of lower bounds though it is flexible for proofs of determinacy from large cardinals. 
 
 \textit{In this paper, ``mouse" is used for Mitchell-Steel models and thus, mice of this paper can only have large cardinals in the order of superstrong cardinals. In inner model theory, mice, which are a type of extender models, are organized in Jensen's hierarchy $\mathcal{J}_\a^{\vec{E}}$. However, in this paper, we will refrain from using the $\mathcal{J}$ notation and will always assume that our $L[\vec{E}]$'s are fine structural unless specified otherwise.}

Nowadays, among other things, methods from inner model theory have been used  (i) to obtain lower bounds on the consistency strength of various combinatorial statements, (ii) to prove the determinacy of various games and (iii) to investigate the structure of models of $AD^+$. 1 of the Main Theorem is an instance of (i). Moreover, the current literature is full of results of this nature and for some such results we refer the reader to \cite{Schim}, \cite{CMI} and \cite{PFA}. Part 4 of the Main Theorem is an example of (ii), and also \cite{DLG} contains many results on the determinacy of long games. \rthm{hod theorem} is an instance of (iii), and also \cite{OIMT} has several results of this nature. Inner model theoretic constructions have also been used to justify the use of large cardinals in set theory and mathematics in general. We refer the reader to  \cite{NewAxioms}, \cite{Maddy1}, \cite{Maddy2}, \cite{TripleHelix}, and \cite{SearchV} for the details of the ongoing debate. Readers who are interested in the history of the inner model problem can consult one of the following excellent sources \cite{INLC}, the introduction of \cite{IT}, \cite{MINLC}, \cite{ABC} and \cite{OIMT}. 

In the next few subsections our goal will be to make some of the notions introduced here more precise. We start with extenders. 

\subsection{Extenders and mice}\label{extenders}

Extenders were introduced in order to capture large fragments of the universe in the ultrapower. For example, if $\mu$ is a $\kappa$-complete non-principal ultrafilter on $\kappa$ then $\powerset(\kappa)\subseteq Ult(V, \mu)$ but in general, $\powerset(\kappa^+)\not \subseteq Ult(V, \mu)$\footnote{In \cite{Cummings}, Cummings showed that it is consistent that there is $\mu$ on $\kappa$ such that $\powerset(\kappa^+) \subseteq Ult(V, \mu)$.}. Using extenders, however, one can form ultrapowers that contain $\powerset(\kappa^+)$, $\powerset(\k^{++})$ and in general, $\powerset(\l)$ for any given $\lambda\geq \k$. Because of such capturing any large cardinal notion can be defined in terms of the existence of certain extenders. We will only introduce the so-called \textit{short} extenders as \textit{long extenders} are only needed to capture large cardinals beyond \textit{superstrong cardinals}, and we will not consider them here. 

Suppose $M$ is a transitive model of some fragment of ZFC and $j: M\rightarrow N$ is an elementary embedding such that $j\not = id$ and the wellfounded part of $N$ is transitive. We let $\cp(j)$ be the least ordinal $\k$ such that $j(\k)\not =\k$. It can be shown that $\k$ is a cardinal in $M$ and $j(\kappa)>\k$. Let then $\k=\cp(j)$ and let $\lambda\geq\kappa$ be in the wellfounded part of $N$ such that $j(\k)\geq \l$. \textit{It is important to note that $N$ need not be wellfounded.} Let then
 \begin{center}
 $E_j=\{ (a, A) : $ for some $n\in \omega$, $a\in \l^n, A\subseteq \k^n$, $A\in M$ and $a\in j(A)\}$.
 \end{center}
We say that $E_j$ is the $(\kappa, \lambda)$ pre-extender over $M$ derived from $j$. $\k$ is called the \textit{critical point} of $E_j$ and $\l$ is called the \textit{length} of $E_j$. 

In literature extenders are defined abstractly and independently of an elementary embedding. However, it can be shown that all extenders are derived from some embedding $j$, namely \textit{the ultrapower embedding}. Because of this, we say $E$ is a $(\k, \l)$ pre-extender over $M$ if there is $j, N$ as above such that $E$ is the $(\k, \l)$ pre-extender over $M$ derived from $j$. We may also say that $E$ is an $M$ pre-extender, emphasizing that $E$ measures sets in $M$. Given a $(\k, \l)$ pre-extender $E$ over $M$, we write $\cp(E)=\k$ and $lh(E)=\l$. Also for $a\in \k^n$, we let $E_a=\{ A\subseteq \k^n : (a, A)\in E\}$. It is then easy to see that for each such $a$, $E_a$ is a $\kappa$-complete $M$-ultrafilter over $\k^n$ and $E_a$ is non-principal iff $a\not \in \k^n$. 

As usual, given a $(\kappa, \lambda)$ pre-extender $E$ over $M$ it is possible to construct $Ult(M, E)$, the ultrapower of $M$ by $E$. First let $A=\{ (a, f): a\in \l^n \wedge f: \k^{\card{a}}\rightarrow M\wedge f\in M\}$. For $a\in \l^n$ and $b\in \l^m$ such that $a\subseteq b$, $a=\la a_0, ..., a_{n-1}\ra$ and $b=\la b_0,..., b_{m-1}\ra$, we define  $\pi_{b, a} : \kappa^{\card{b}} \rightarrow \kappa^{\card{a}}$, a projection of $\k^{\card{b}}$ into $\k^{\card{a}}$, as follows: given $s=(s_0, ..., s_{m-1})$,
 \begin{center}
  $\pi_{b, a}(s)=(s_{i_0}, s_{i_1},..., s_{i_{n-1}})$
 \end{center}
where $\la i_{k} : k<n \ra$ is such that $b_{i_k}=a_k$. Given $f: \kappa^{\card{a}}\rightarrow M$, we let $f^{a, b}:\kappa^{\card{b}}\rightarrow M$ be defined by
\begin{center}
$f^{a, b}(s)=f(\pi_{b, a}(s))$.
\end{center}
Then $=_E$ on $A$ is defined as expected, namely,
\begin{center}
$(a, f)=_E (b, g) \Leftrightarrow \{ s: f^{a, a\cup b}(s)= g^{b, \a\cup b}(s)\}\in E_{a\cup b}$.
\end{center}
Let $[a, f]$ be $=_E$-equivalence class of $(a, f)$ and let $D=\{ [a, f] : (a, f)\in A\}$. The relation $\in_E$ is now defined on $D$ by
\begin{center}
$(a, f)\in_E (b, g) \Leftrightarrow \{ s: f^{a, a\cup b}(s)\in g^{b, a\cup b}(s)\}\in E_{a\cup b}$.
\end{center}
We call $(D, \in_E)$ the ultrapower of $M$ by $E$ and denote it by $Ult(M, E)$. For each $x$, let $c_x:\kappa\rightarrow M$ be the function $c_x(\a)=x$ and define $j_E: M\rightarrow Ult(M, E)$ by
 \begin{center}
$j_E(x)=[\kappa, c_x]$.
 \end{center}
It can be shown,  by proving Los' lemma, that $j_E$ is an elementary embedding called the \textit{ultrapower embedding by $E$}. It can also be shown that $\cp(j_E)=\k$, $j_E(\kappa)\geq \l$ and $E=E_{j_E}$. 

An important remark is that $Ult(M, E)$ may not be wellfounded. If it is wellfounded then we identify it with its transitive collapse. If, however, $j: M\rightarrow N$ is such that $E=E_j$ and $N$ is wellfounded then $Ult(N, E)$ is also wellfounded. This is because it is not hard to show that if $\sigma:Ult(M, E)\rightarrow N$ is defined by $\sigma([a, f])=j(f)(a)$ then $\sigma$ is elementary and  $j=\sigma \circ j_E$.  If $Ult(M, E)$ is wellfounded then we say that $E$ is a $(\kappa, \lambda)$-extender over $M$. Wellfoundness of ultrapowers is an important issue in inner model theory and it is not always the case that ultrapowers that we would like to form are wellfounded.  

Essentially any large cardinal axiom can be stated using extenders. 
Because of this, it is natural to look for canonical structures with large cardinals among the \textit{extender models}, which are models constructed from a sequence of extenders. \textit{Premice} and \textit{mice}, which are evolved versions of premice, are such structures. Both are \textit{extender models} of the form $L_\a[\vec{E}]$ where $\a$ is either an ordinal or $\a=Ord$ and $\vec{E}$ is a \textit{coherent sequence of extenders}. If $L_\a[\vec{E}]$ is a premouse and $\b\in dom(\vec{E})$ then $E_\b$ is an extender over $L_\b[\vec{E}\rest \b]$ and in general, may not be an extender over a longer initial segment of $L_\a[\vec{E}]$. 
\textit{Iterability} is what makes premice canonical, and mice are \textit{iterable} premice.   

It is not hard to see that every set of ordinals can be coded into some model constructed from a (non-coherent) sequence of extenders. Given an arbitrary set of ordinals $A\subseteq \beta$, we would like to construct $\vec{E}$ such that $A\in L[\vec{E}]$. There are two ways of doing it. 
\begin{enumerate}
\item  (Repeating an extender) Fix an extender $E$ and let $(\vec{E})_\a=E$ if $\a\in A$ and otherwise, let $(\vec{E})_\a$ be undefined. Then $A\in L[\vec{E}]$.
\item (Skipping an initial segment) Again, fix an extender $E$ with the property that there is an increasing sequence $\la \l_\a: \a<\b\ra$ cofinal in the length of $E$.  For $\a<\b$, let $E_\a=\{ (a, A)\in E: a\in \lambda_\a^{<\omega}\}$. Define $\vec{E}$ by setting $(\vec{E})_\a=E_\a$ if $\a\in A$ and otherwise, let $(\vec{E})_\a$ be undefined. Then $A\in L[\vec{E}]$.
\end{enumerate}  
As set of ordinals can code any set, using either way of coding we can reorganize $V$ as $L[\vec{E}]$. It follows that in order for $L[\vec{E}]$ to contain only canonical information, $\vec{E}$ must have special properties and coherent sequence of extenders do have very special properties. The exact definition of $\vec{E}$ that yields canonical structures can be found in \cite{OIMT} and we will not state it here. Instead, we explain how to block the two ways of coding sets into $L[\vec{E}]$. 

Repeating an extender is blocked by a special way of indexing the extenders. Mitchell-Steel extender sequences use the successors of the generators to index extenders. Given an $M$-extender $E$ and $\eta<lh(E)$, we let $E|\eta=\{ (a, A)\in E : a\in \eta^{<\omega}\}$. We say $\eta<lh(E)$ is a \textit{generator} of $E$ if letting $\sigma: Ult(M, E|\eta)\rightarrow Ult(M, E)$ be the embedding given by
$\sigma([s, f]_{E|\eta})=[s, f]_{E}$, $\cp(\sigma)=\eta$. Let then $\nu_E=\sup\{ \eta+1 : \eta$ is a generator of $E\}$.  Let $\a=(\nu_E^{+})^{Ult(M, E)}$. Then $E$ is indexed at $\a$. Thus, no extender can be repeated. 

\textit{In the above explanation we skipped over a technical point. In general, what one indexes at $\a$ is the trivial completion of $E$. But the exact details are irrelevant here.} There is also Friedman-Jensen indexing in which $E$ is indexed at $j_{E}(\kappa^+)$. However, in this paper we will deal with only Mitchell-Steel way of indexing extenders. Finally,
skipping is blocked by the \textit{initial segment condition} which essentially says that if $E$ is on the sequence then all relevant initial segments of $E$ are also on the sequence. ``Coherent" refers to the following property: given a premouse $L_\a[\vec{E}]$ and $\b\in dom(\vec{E})$, we have that $L_\b[\vec{E}\rest \b]=L_\b[j_{E_\b}(\vec{E})\rest \b]$. We will explain iterability In the next subsection. Below is a summary of the large cardinal notions that will appear in this paper. 

\begin{enumerate}
\item (Strong cardinals) $\kappa$ is \textit{$\lambda$-strong} if there is a $(\kappa, \lambda)$-extender $E$ such that $V_\l\subseteq Ult(V, E)$ and $j_E(\kappa)\geq \l$. $\k$ is $<\l$-strong if for every $\eta<\l$, $\kappa$ is $\eta$-strong.  $\kappa$ is \textit{strong} if it is $\lambda$-strong for every $\l$.
\item (Superstrong cardinals) $\kappa$ is \textit{superstrong} if there is some $\l\geq \k$ and a $(\kappa, \lambda)$-extender $E$ such that $V_\l\subseteq Ult(V, E)$ and $j_E(\kappa)=\l$.
\item (Woodin cardinals)  A regular cardinal $\d$ is \textit{Woodin} if for every $f:\d\rightarrow \d$ there is $\k<\d$ and extender $E$ with critical point $\k$ such that $V_{j_E(f)(\k)}\subseteq Ult(V, E)$.  $\d$ is a \textit{Woodin limit of Woodin cardinals} if it is a Woodin cardinal and is also a limit of Woodin cardinals. 
\item (Supercompact cardinals) $\k$ is called a supercompact cardinal if for every $\l\geq \k$ there is $j: V\rightarrow M$ such that $\cp(j)=\k$, $M^\lambda\subseteq M$ and $j(\k)>\l$. 
\end{enumerate}

The following proposition relates the large cardinals introduced above. Its proof can be found in \cite{Jech} or in \cite{Kanamori}.

\begin{proposition}\label{large cardinals} Suppose $\k$ is a cardinal. Then the following holds.
\begin{enumerate}
\item If $\k$ is a Woodin cardinal then it is a limit of $<\k$-strong cardinals.
\item If $\k$ is a superstrong cardinal then $\kappa$ is a Woodin limit of Woodin cardinals but may not be a strong cardinal.
\item If $\k$ is a supercompact cardinal then $\k$ is a limit of superstrong cardinals but it may not be a superstrong cardinal.
\end{enumerate}
\end{proposition}

\subsection{Simple iterations and mice}\label{iterability and mice}

Intuitively, a premouse $\M$ is iterable if all the meaningful ways of taking ultrapowers and direct limits by first using an extender from $\M$ and continuing from there results in well-founded models. In early days of inner model theory, when premice could only have very restricted large cardinals, iterations were all \textit{linear}. 

We say $\M=(M, \nu, \in)$ is an \textit{$L[\mu]$-like premouse} if $M$ is transitive and there is an $\M$-cardinal $\k$ such that $M\models ``\nu$ is a non-principal normal $\k$-complete ultrafilter over $\kappa$ and $V=L[\nu]$" (in particular, $\nu\in M$). We let $\k^\M$ witness the above sentence in $M$ and let $\mu^\M=\nu$. Also, because $M$ is transitive, there is some $\b$, denoted by $\b^\M$, such that $\M=L_\b[\nu]$. Notice that $L[\mu]$-like premoushood is a first order property. 

\begin{definition}
Suppose $\M=L_\b[\nu]$ is $L[\mu]$-like premouse and $\l$ is some ordinal. An iteration of $\M$ of length $\lambda$ is a sequence $\la \M_\a, \nu_\a, i_{\a,\b} : \a<\b<\lambda\ra$ defined by induction on $\b$ as follows: 
 \begin{enumerate}
 \item $\M_0=\M$ and $\nu_0=\nu$,
 \item for every $\a<\b$, $\M_\a$ is an $L[\mu]$-like premouse such that $\nu^{\M_\a}=\nu_\a$,
 \item for all $\a<\b$, $\M_{\a+1}=Ult(\M_\a, \nu_\a)$ and $i_{\a, \a+1}=j_{\nu_\a}:\M_{\a}\rightarrow \M_{\a+1}$ is the ultrapower embedding,
 \item if $\b=\gg+1$ and $\a\leq \gg$ then $i_{\a, \b+1}=i_{\b, \b+1}\circ i_{\a, \gg}$ and $\nu_{\b}=i_{0, \b}(\nu)$
 \item if $\b$ is a limit ordinal then $\M_\b$ is the direct limit of $\la \M_\gg : \gg<\b\ra$ under $i_{\gg, \xi}$'s for $\gg<\xi<\b$, $i_{\a, \b}$ is the direct limit embedding and $\nu_\b=i_{0, \b}(\nu)$.
 \end{enumerate}

 \end{definition}
 
 Notice that the $\l$-iteration of $\M$, if exists, is unique. However, it may not exists as some $\M_\b$ for $\b<\l$ could be ill-founded implying the clause 2 cannot be satisfied. Notice that it follows from the elementarity of $i_{0, \a}$ that if $\M_\a$ is well-founded then it is $L[\mu]$-like premouse. 
  
 \begin{definition}
 Suppose $\M=L_\b[\nu]$ is $L[\mu]$-like premouse and $\l$ is some ordinal. $\M$ is called $\lambda$-iterable if its length $\l$ iteration exists. $\M$ is called \textit{iterable} if it is $\l$-iterable for all $\l$. We say $\M$ is an $L[\mu]$-like mouse if it is an iterable $L[\mu]$-like premouse.
 \end{definition}

It cannot be shown in ZFC alone that there is an iterable $L[\mu]$-like mouse: for instance, such a mouse cannot exist in $L$ (see \cite{Jech} or \cite{Kanamori}). Iterability of $L[\mu]$-like premice that contain the ordinals is a theorem of ZFC.

\begin{theorem}[Gaifman, \cite{Jech}]\label{gaifman} Suppose $\M$ is an $L[\mu]$-like premouse and $Ord\subseteq \M$. Then $\M$ is an $L[\mu]$-like mouse. 
\end{theorem}

However, iterability of set size $L[\mu]$-like mice is \textit{not} a theorem of ZFC. While iterability has many applications, perhaps its most important application is \textit{comparison}. Given an $L[\mu]$-like premouse $\M$, we say $\N$ is the $\l$th iterate of $\M$ if length $\l+1$ iteration of $\M$ exists and $\N$ is the $\l$th model of this iteration.

\begin{definition}\label{comparison of l[mu] like mice} Given two $L[\mu]$-like premice $\M$ and $\N$, we say comparison holds for $\M$ and $\N$ if there are $\xi$ and $\eta$ such that if $\P=L_\a[\mu]$ and $\Q=L_\b[\nu]$ are respectively the $\xi$th iterate of $\M$ and the $\eta$th iterate of $\N$ then $\k^\P=\k^\Q$ and either
\begin{enumerate}
\item $\b\leq \a$ and $\Q=L_\b[\mu]$ or
\item $\a\leq \b$ and $\P=L_\a[\nu]$. 
\end{enumerate} 
\end{definition} 

 Kunen established comparison for $L[\mu]$-like mice. The proof for class size $L[\mu]$-like mice gives much more. 

\begin{theorem}[Kunen, \cite{Jech}]\label{comparison of simple mice} Suppose $\M$ and $\N$ are two $L[\mu]$-like mice. Then comparison holds for $\M$ and $\N$. Moreover, if $Ord \subseteq \M\cap \N$ then assuming $\k^\M\leq \k^\N$, there is some $\l$ such that if $\la \M_\a, \nu_\a, i_{\a, \b} : \a<\b<\l+1\ra$ is the iteration of $\M$ of length $\l+1$ then $\M_{\l}=\N$. In particular, if $\k^\M=\k^\N$ then $\M=\N$.
\end{theorem}

Comparison has many important consequences. It is the reason that mice are considered canonical models. For instance, comparison of $L[\mu]$-like mice implies that they all have the same reals. In fact, more is true. Given a premouse $\M$, we let $\leq_\M$ be its constructibility order.  

\begin{theorem}[Silver, \cite{Kanamori}]\label{silvers theorem} Suppose $\S$ is a class size $L[\mu]$-like mouse. Then $\mathbb{R}\cap \S$ and $\mathbb{R}^2\cap \leq_\S$ are $\Sigma^1_3$.
\end{theorem}
\begin{proof} We only sketch the argument. The full proof can be found in \cite{Jech}. The  \textit{the Dodd-Jensen} ordering $\leq_{DJ}$ of $L[\mu]$-like mice is defined by letting $\M\leq_{DJ}\N$ if clause 2 of \rdef{comparison of l[mu] like mice} holds. It follows from comparison that if $\M$ and $\N$ are two $L[\mu]$-like mice such that $\M\leq_{DJ} \N$ then $\mathbb{R}^2\cap \leq_\M$ is an initial segment of $\mathbb{R}^2\cap \leq_\N$.  

Notice that, using a standard Skolem hull argument, it can be shown that for countable $L[\mu]$-like premice iterability is equivalent to $\omega_1$-iterability. It then follows that the statement that an $L[\mu]$-like $\M$ is iterable is equivalent to a $\Pi^1_2$-statement, namely that for every countable $\a$ the length $\a$-iteration of $\M$ exists. We then claim that
\begin{center}
$x\in \S$ iff there is a countable $L[\mu]$-like mouse $\M$ such that $x\in \M$.
\end{center}
Notice that the statement on the right side is indeed $\Sigma^1_3$. Also, the right-to-left direction follows immediately from comparison. To show the left-to-right direction, let $x\in \mathbb{R}\cap \S$. Let $\k=\k^\S$, $\l=(\k^{+\omega})^\S$ and $\pi:\M\rightarrow L_\l[\mu^\S]$ be a countable Skolem hull of $L_\l[\mu^\S]$ such that $x\in \M$. Notice that $\M$ is $L[\mu]$-like premouse. To finish, it is enough to show that $\M$ is $\omega_1$-iterable. 

To prove iterability of $\M$ we use what's called \textit{the copying construction}. Such constructions are used to show that if $\P$ is iterable and $\pi:\Q\rightarrow \P$ is an elementary embedding then $\Q$ is iterable. Given an iteration $q$ of $\Q$, the copying construction produces an iteration $p$ of $\P$ such that all the models appearing in $q$ are embedded into some model of $p$. Since models appearing in $p$ are wellfounded, this implies that the models appearing in $q$ are wellfounded as well. 

To perform the copying construction in our situation, first let $\N=L_\l[\mu^\S]$. Fix some $\b<\omega_1$ and let $\la \M_\xi, \nu_\xi, i_{\xi,\eta}: \xi<\eta<\b\ra$ be the iteration of $\M$ of length $\b$. We produce a sequence $\la \N_\xi, \mu_\xi, j_{\xi, \eta}, \pi_\xi: \xi<\eta<\b\ra$ such that 
\begin{enumerate}
\item $\la \N_\xi, \mu_\xi, j_{\xi, \eta}: \xi<\eta<\b\ra$ is the length $\b$ iteration of $\N$,
\item $\pi_\xi:\M_\xi\rightarrow \N_\xi$,
\item $\pi_0=\pi$,
\item for all $\xi<\eta<\b$, $j_{\xi, \eta}\circ\pi_{\xi}=\pi_\eta \circ i_{\xi, \eta}$.
\end{enumerate}
The construction is by induction and the first step of the induction has already been taken care of. Also, the construction at limit steps just comes from the direct limit constructions and we leave the details to the reader. Suppose now that we have constructed $\la \N_\xi, \mu_\xi, j_{\xi, \eta}, \pi_\xi: \xi<\eta\leq \gg \ra$  for some $\gg$ such that $\gg+1<\b$. We need to construct $\N_{\gg+1}$ and $\pi_{\gg+1}$. We have that $\pi_\gg:\M_\gg\rightarrow \N_\gg$ and that $\pi_\gg(\nu_\gg)=\mu_\gg$. We also have that $\M_{\gg+1}=Ult(\M_\gg, \nu_\gg)$. Then let $\N_{\gg+1}=Ult(\N_\gg, \mu_\gg)$, $j_{\gg, \gg+1}:\N_\gg\rightarrow \N_{\gg+1}$ and let $\pi_{\gg+1}:\M_{\gg+1}\rightarrow \N_{\gg+1}$ be given by the following formula: for $x=i_{\gg, \gg+1}(f)(\k^{\M_\gg})\in \M_{\gg+1}$,
\begin{center}
$\pi_{\gg+1}(x)=j_{\gg, \gg+1}(\pi_\gg(f))(\k^{\N_\gg})$.
\end{center}
It is not hard to see that $\pi_{\gg+1}$ is as desired. 

It follows from \rthm{gaifman} that $\N$ is iterable. It then also follows that $\M$ is iterable as any iterate of $\M$ is embedded into an iterate of $\N$. This finishes our outline of the proof of \rthm{silvers theorem}.
\end{proof}

 Given two premice $\M=L_\a[{\vec{E}}]$ and $\N=L_\b[{\vec{F}}]$ we write $\M\trianglelefteq\N$ if $\a\leq \b$ and $\vec{E}=\vec{F}\rest \a+1$. We say \textit{comparison holds} for $\M$ and $\N$ if there is an \textit{iterate} $\P$ of $\M$ and an \textit{iterate} $\Q$ of $\N$ such that either $\P\inseg \Q$ or $\Q\inseg \P$. Can linear iterations, i.e., iterations where the next model in the iteration is an ultrapower of the previous one, be used to prove comparison of arbitrary mice? Suppose the answer is yes. Notice that if comparison holds for $\M$ and $\N$ then $\mathbb{R}^2\cap \leq_\M$ is compatible with $\mathbb{R}^2\cap \leq_\N$. In particular, if linear iterations are enough for proving a general comparison theorem, then, as in the proof of \rthm{silvers theorem}, the set
 \begin{center}
 $\{ x\in \mathbb{R}: x$ is in some mouse$\}$
 \end{center}
is $\Sigma^1_3$ and has a $\Sigma^1_3$-wellordering. It then would follow that any mouse satisfying ``$\omega_1$ exists" has a $\Sigma^1_3$-wellordering of its reals. However, Foreman, Magidor and Shelah, in \cite{MartinMax}, showed that supercompact cardinals imply that there is no  $\Sigma^1_n$ wellordering of the reals. It must then be that mice with significant large cardinals cannot be compared using linear iterations. Martin and Steel, then, showed that the new complexity comes from Woodin cardinals and also isolated a new form of iterability that led to the full proof of comparison at the level of superstrong cardinals.

\subsection{Iterability and mice}

Comparison of premice with significant large cardinals needs more complicated forms of iterability. In particular, comparison of premice with Woodin cardinals cannot be done via simple linear iterations as one gets into the so-called \textit{moving generators problem}, which we will not explain here but see \cite{IT} or \cite{OIMT}. In \cite{IT}, Martin and Steel defined iterability in terms of a game, the \textit{iteration game}, and in \cite{FSIT}, Mitchell and Steel used this form of iterability to prove the general comparison lemma. Below we sketch the iteration game.

First we remark that if $M$ and $N$ are transitive models of set theory, $E\in N$ is a $(\kappa, \lambda)$-extender in $N$ and $\powerset(\kappa)^M=\powerset(\k)^N$ then one can define $Ult(M, E)$ in the usual way by using all the functions that are in $M$. In this case we say that \textit{it makes sense to apply} $E$ to $M$ or that $E$ \textit{can be applied to} $M$.

$T$ is said to be a \textit{strict tree order} on $\a$ if $T$ is a strict partial order on $\a$ such that
\begin{enumerate}
\item $\b T\gg \Rightarrow \b<\gg$,
\item for every $\gg$, $\{ \b : \b T\gg \}$ is wellordered by $T$,
\item $\gg$ is a successor ordinal $\Leftrightarrow $ $\gg$ is $T$-successor, and
\item $\gg$ is a limit ordinal $\Rightarrow \{ \b : \b T \gg\}$ is closed and $\in$-cofinal in $\gg$.
\end{enumerate}
If $T$ is a strict partial order on $\a$ then we let $pred_T$ be the predecessor function on $T$.

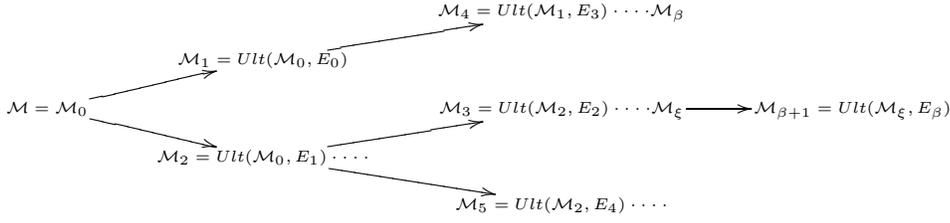
\begin{figure}
\begin{center}\tiny{
$\xymatrix@1@R=10pt{
& & & & & & & & & & & & & && & & & & & &\\
& &\M_4=Ult(\M_1, E_3) \cdot\cdot\cdot\cdot \M_\b& & & & & & & & & & && & & & & & &\\
& \M_1=Ult(\M_0, E_0) \ar[ur] & &  & & & & & & \\
\M= \M_0 \ar[ur] \ar[dr] & & \M_3=Ult(\M_2, E_2) \cdot\cdot\cdot\cdot  \M_\xi \ar[r]&\M_{\b+1}=Ult(\M_\xi, E_\b) & & & &  \\
& \M_2=Ult(\M_0, E_1) \ar[ur] \ar[dr] \cdot\cdot\cdot\cdot & & & & & & & \\
& & \M_5=Ult(\M_2, E_4) \cdot\cdot\cdot\cdot&  & & & & & & & & & & & & & & & & & &\\
& & & & & & & & & & & & & && & & & & & &}$}
\end{center}
\caption{An iteration tree.}
\label{iteration tree}
\end{figure}

We advise the reader to consult \rfig{iteration tree} while reading the description of the iteration game. The iteration game $\mathcal{G}_\k(\M)$ on a premouse $\M$ is a two player game of length $\k$ in which the players construct a tree of models such that each successor node on the tree is obtained by an ultrapower of a model that already exists in the tree. 

As the players play the game they produce a strict tree order $T$ on $\k$, a sequence of models $\la \M_\a : \a<\k\ra$, and a sequence of extenders $\la E_\a : \a<\k\ra$ such that $E_\a$ is an extender appearing on the sequence of extenders of $\M_\a$, i.e., if $\M_\a=L_\b[\vec{F}]$ then for some $\gg$, $\vec{F}(\gg)=E_\a$. Player $I$ plays all successor stages and player $II$ plays at limit stages. 

Suppose at stage $\gg+1$ the players have constructed $\la \M_\a, E_\a :\a\leq \gg\ra$ and $T\rest \gg+1$. Then $I$ chooses an extender $E_\gg$ from the extender sequence of $\M_\gg$ such that for every $\a<\gg$, $lh(E_\a)<lh(E_\gg)$. Let $\b\leq \gg$ be the least $\xi$ such that $\cp(E_\gg)\in [\cp(E_\xi), lh(E_\xi))$. It can be shown that it makes sense to apply $E_\gg$ to $\M_\b$. Player $I$ then sets
\begin{center}
$\b=pred_T(\gg+1)$ and $\M_{\gg+1}=Ult(\M_\gg, E_\gg)$.
\end{center}

Suppose now that $\gg$ is a limit ordinal and the players have constructed a strict tree order $T$ on $\gg$ and a sequence $\la \M_\a, E_\a: \a<\gg\ra$. We then also have a sequence of commuting embeddings $i_{\a, \b}:\M_\a \rightarrow \M_\b$ such that
\begin{enumerate}
\item $i_{\a, \b}$ is defined whenever $\a T \b$,
\item if $\a=pred_T( \b+1)$ then $i_{\a, \b+1}=i_{E_\b}: \M_\a\rightarrow Ult(\M_\a, E_\b)$,
\item $\a T \xi T \b \Rightarrow i_{\a, \b}=i_{\xi, \b}\circ i_{\a, \xi}$,
\item if $\xi<\gg$ is a limit then $\M_\xi$ is the direct limit of $\{ \M_\a : \a T \xi\}$ under the $i_{\a, \b}$'s, and $i_{\a, \xi}$ is the direct limit embedding.
\end{enumerate}
Now it is $II$'s turn to move and she chooses a cofinal branch $b$ through $T$ and lets $\M_\gg$ be the direct limit of $\{ \M_\a : \a\in b\}$ under the $i_{\a, \b}$'s. We then let $\xi T \gg \Leftrightarrow \xi \in b$. The game lasts $\k$-steps and $II$ wins if all $\M_\a$'s are wellfounded.

\textit{We skip over a technical point here. Oftentimes to form $\M_{\b+1}$, one needs to drop to a certain initial segment of $\M_\xi$ and take the ultrapower of that initial segment. Because this article is of expository nature, we will overlook this point whenever it comes up. See \cite{OIMT} for a more detailed description of the iteration game.}

A run of the iteration game on $\M$ in which neither player has lost is called an \textit{iteration tree} on $\M$. An iteration tree $\T$ on $\M$ has the form $\T=\la T, \M_\a^\T, E_\a^\T : \a< lh(\T)\ra$ where $T$ is the tree structure of $\T$. In particular, in \rfig{iteration tree}, $\xi$ is the predecessor of $\b+1$. We write this as $\xi=pred_T(\b+1)$. If $\a<_T\b$ then there is an embedding $i_{\a, \b}^\T:\M^\T_\a\rightarrow \M_\b^\T$. If $lh(\T)=\xi+1$ then we let $i^\T=_{def}i^\T_{0, \xi}:\M_0\rightarrow \M_\xi^\T$. $i^\T$ is called the iteration embedding. Also, if $lh(T)$ is a limit ordinal, $b$ is a cofinal branch through $\T$ and $\a\in b$ then $i_{\a, b}^\T:\M_\a^\T\rightarrow \M_b^\T$ is the direct limit embedding where $\M_b^\T$ is the direct limit of $\la \M_\xi^\T : \xi \in b\ra$ under the $i^\T_{\xi, \nu}$'s. If $\a=0$ then we let $i^\T_b=_{def}i_{0, b}^\T$.

\begin{definition}
A premouse $\M$ is called $\kappa$-iterable if $II$ has a winning strategy in $\mathcal{G}_\k(\M)$.  $\Sigma$ is called a $\kappa$-iteration strategy for $\M$ if it is a winning strategy for $II$ in $\mathcal{G}_\k(\M)$.
\end{definition}

How iterable does a premouse need to be in order to be called a mouse? In general, to have a good theory of mice one needs \textit{countable iterability}. We say a premouse $\M$ is countably iterable if whenever $\pi:\N\rightarrow \M$ is such that $\pi$ is elementary and $\N$ is countable and transitive then $\N$ is $\omega_1+1$-iterable.

However, in this paper, we will mostly consider premice that are countable and in this case countable iterability is equivalent to $\omega_1+1$-iterability. Thus, we say $\M$ is a mouse if it is $\omega_1+1$-iterable\footnote{It follows from the copying constructions used in \rthm{silvers theorem} that, in general, $\omega_1+1$-iterability implies countable iterability.}.  Often we will work under the \textit{Axiom of Determinacy} ($AD$) which implies that $\omega_1$ is a measurable cardinal. It then follows that under $AD$, $\omega_1+1$ iterability for countable mice is equivalent to $\omega_1$-iterability. Because of this, while working under $AD$, we will mostly deal with $\omega_1$-iterability. The importance of this move is that $\omega_1$-iteration strategies can be coded by sets of reals while $\omega_1+1$-iteration strategies cannot. When $\Sigma$ is an $\omega_1$-iteration strategy for some countable mouse then we let $Code(\Sigma)$ be the set of reals coding $\Sigma$. Such a coding is done relative to some fixed (and standard) way of coding hereditarily countable sets by reals.  

\subsection{Comparison}

 Comparison for arbitrary mice takes the following form. Given a premouse $\M$, we say $\P$ is an iterate of $\M$ if there is an iteration tree $\T=\la T, \M_\a^\T, E_\a^\T: \a<\eta+1\ra$ on $\M$ such that $\M^\T_\eta=\P$. If $\Sigma$ is an iteration strategy for $\M$ then we say $\P$ is a $\Sigma$-iterate of $\M$ if $\T$ above can be take to be according to $\Sigma$. 
 
 \begin{definition}[Comparison] Given two premice $\M$ and $\N$, we say comparison holds for $\M$ and $\N$, if there is an iterate $\P$ of $\M$ and an iterate $\Q$ of $\N$ such that either $\P\inseg \Q$ or $\Q\inseg \P$. We say $(\P, \Q)$ witness the comparison of $(\M, \N)$.
 \end{definition} 
 
 \begin{theorem}[Mitchell-Steel, \cite{FSIT}]\label{comparison of mice} Suppose $\M$ and $\N$ are two mice which are
 $\max(\card{\M}^+,\\ \card{\N}^+)+1$-iterable via iteration strategies $\Sigma$ and $\Lambda$ respectively. Then there is a $\Sigma$-iterate $\P$ of $\M$ and a $\Lambda$-iterate $\Q$ of $\N$ such that $(\P, \Q)$ witness the comparison of $(\M, \N)$. In particular,  the comparison holds for $\M$ and $\N$.
  \end{theorem}
  
 The proof of \rthm{comparison of mice} uses the \textit{comparison process} which is a systematic way of removing the extenders that cause disagreements between $\M$ and $\N$. 
  
\begin{corollary}
  Suppose $\M$ and $\N$ are two mice which are  $\max(\card{\M}^+, \card{\N}^+)+1$-iterable. Then $\mathbb{R}^2\cap \leq_\M$ is compatible with $\mathbb{R}^2\cap \leq_\N$.
\end{corollary}
  
Notice that our current form of iterability doesn't imply that 
  \begin{center}
  $\{ x \in \mathbb{R}: x$ is in some mouse$\}$
  \end{center} 
is $\Sigma^1_n$ for any $n\in \omega$. In general, the following is the best we can do. 
  
  \begin{corollary}[\cite{OIMT}]\label{complexity of reals}
  The set $\{x\in \mathbb{R}: x$ is in some mouse$\}$ is $\Sigma^2_3$. Moreover, if a real $x$ is in some mouse then it is $\Delta^2_2$ in some countable ordinal\footnote{I.e, there is some countable ordinal $\a$ and a $\Delta^2_2$ formula $\phi$ such that for any code $z$ of $\a$, $n\in x\iff \phi(n, z)$.}.
  \end{corollary}
  \begin{proof} We have that $x$ is in a mouse iff there is some countable premouse $\M$ such that $x\in \M$ and there is $A\subseteq \mathbb{R}$ such that $A$ codes an $\omega_1+1$-iteration strategy for $\M$. The statement ``$A$ codes an $\omega_1+1$-iteration strategy" is $\Pi^2_2(A)$ as we need to say that for every iteration tree $\T$ on $\M$ of height $\omega_1$ according to the strategy coded by $A$ there is a branch. The proof of the second part of the lemma uses the exact logical complexity of the comparison process and can be found in \cite{OIMT}.
  \end{proof}
 
 Finally, in \cite{AbSh}, Abraham and Shelah showed that all large cardinals are compatible with $\Sigma^2_2$ wellordering of the reals implying that aforementioned result of Foreman, Magidor and Shelah isn't problematic anymore.

\subsection{The Levy hierarchy and mice}\label{levy hierarchy and mice}

An immediate consequence of \rcor{complexity of reals} is that reals appearing in mice are ordinal definable. In many cases, a weakening of the converse of this fact is true as well. In particular, reals definable at various levels of the Levy hierarchy are exactly those that appear in certain kind of mice. For instance, a real is $\Delta^1_1$ iff it is hyperarithmetic or equivalently, is in $L_{\omega_1^{ck}}$ (see \cite{Sacks}). A real is $\Delta^1_2$ in a countable ordinal iff it is in $L$ (see \cite{Kanamori}). If $\Pi^1_2$-determinacy holds then a real is $\Delta^1_3$ in a countable ordinal iff it is in the minimal proper class mouse with a Woodin cardinal (Martin, Steel and Woodin, \cite{PWIM}). The list of such results goes on.

The second one of the above results is a consequence of the celebrated \textit{Shoenfield's absoluteness theorem}. Its proof has several key ideas that are permanent features in many of the arguments in descriptive inner model theory. In particular, the proof of \rcon{the capturing of hod pairs} below $AD_{\mathbb{R}}+``\Theta$ is regular", one of the main theorems proved towards the proof of MSC, is reminiscent of this proof. Because of this, we give a quick sketch of the proof of Shoenfield's absoluteness theorem.

\begin{theorem} $L$ is $\Sigma^1_2$-correct, i.e., for a $\Sigma^1_2$ sentence $\phi$,  
\begin{center}
$V\models \phi \iff L\models \phi$.
\end{center}
\end{theorem}

Suppose $\kappa$ is a cardinal. Recall that $T$ is a tree on $\omega\times \k$ if $T\subseteq \cup_{n<\omega}\omega^n\times\k^n$ and $T$ is closed under initial segments. $[T]$ is the set of branches of $T$, i.e., 
\begin{center}
$[T]=\{ (x, f)\in \omega^\omega\times \kappa^\omega: \forall n\in \omega ((x\rest n, f \rest n) \in T)\}$. 
\end{center}
The projection of $T$, denoted by $p[T]$, is the projection of $[T]$ on its first coordinate, i.e.,
\begin{center}
 $p[T]=\{ x\in \mathbb{R} : \exists f\in \k^\omega (x, f)\in [T]\}$.
\end{center} 
The following is a fundamental yet easy lemma. 

\begin{lemma}\label{capturing truth in inner models} Suppose $T$ is a tree on some $\omega\times\kappa$. Suppose $M$ is an inner model satisfying ZFC such that $T\in M$. Then for any poset $\mathbb{P}\in M$,
\begin{center}
$M\models p[T]\not =\emptyset \iff M^{\mathbb{P}}\models p[T]\not =\emptyset\iff p[T]\not =\emptyset$.
\end{center}
\end{lemma}
\begin{proof}
The proof is based on the fact that if $p[T] =\emptyset$ then there is $f:T\rightarrow Ord$ which is a ranking function.
\end{proof}

\rlem{capturing truth in inner models} is a prescription for showing that various inner models capture fragments of what is true in the universe. In particular, Shoenfield's absoluteness lemma is based on it. 

\begin{lemma} Suppose $A$ is a $\Sigma^1_2$ set of reals. Let $\k\geq \omega_1$ be a cardinal. Then $A$ is $\k$-Suslin as witnessed by a tree $T\in L$. 
\end{lemma}
\begin{proof}
 Fix  a $\Pi^1_1$-set $B$ such that $A$ is the projection of $B$. Notice that if $T\in L$ is a tree on $\omega\times \k$ such that $p[T]=B$ then we can easily design a tree $S\in L$ on $\omega\times \k$ such that $p[S]=A$. Without loss of generality we then assume that $A$ is $\Pi^1_1$. Let $S$ be a recursive tree on $\omega\times \omega$ such that $p[S]=A^c$. The desired tree $T$ is such that if $(x, f)\in [T]$ then letting $S_x=\{ t \in \omega^{<\omega} : (x\rest lh(t), t)\in S\}$ then $f$ is a ranking function for $S_x$. The construction of such a tree $T$ is possible because for each $x\in \omega^{\omega}$ and $n\in \omega$, $S_x\cap \omega^n$ depends on just $x\rest n$. A careful computation shows that $T\in L$.
\end{proof}

\begin{corollary}\label{capturing sigma21}
Suppose $A\subseteq \mathbb{R}$ is a nonempty $\Sigma^1_2$ set. Then $A\cap L\not= \emptyset$. 
\end{corollary}
\begin{proof}
Suppose not. Let $T\in L$ be such that $p[T]=A$. Suppose for a moment that $L\models [T]=\emptyset$. Then it follows from \rlem{capturing truth in inner models} that $p[T]=\emptyset$. Hence, $A=\emptyset$. Thus, it must be the case that $L\models [T]\not=\emptyset$ and therefore, $A\cap L\not =\emptyset$. 
\end{proof}

\begin{corollary}\label{absoluteness}
Suppose $\phi$ is a $\Sigma^1_2$-sentence. Then 
\begin{center}
$\phi \iff L\models \phi$.
\end{center}
\end{corollary}
\begin{proof}
Let $\psi$ be $\Pi^1_1$ such that $\phi=\exists x \psi(x)$. Let $A=\{ x : \psi(x)\}$. Then it follows from \rcor{capturing sigma21} that  $A\not =\emptyset \iff A\cap L\not =\emptyset$. Therefore, $\phi \iff L\models \phi$.
\end{proof}

This finishes our outline of the proof of the Shoenfield's absoluteness theorem. The basic structure of the argument presented above is used to prove that mice with large cardinals are correct for more complicated class of sentences. Often times new techniques, such as \textit{Woodin's genericity iterations}, are needed. Because we will not deal with the exact nature of genericity iterations in this paper, we won't explain them but we will use them in \rsec{dimt}. Genericity iterations appear in many of the arguments in this area of set theory. Interested readers should consult \cite{EA} and \cite{OIMT} for an excellent introduction to genericity iterations. The following theorem due to Woodin can be proved using the above argument coupled with genericity iterations. 

Suppose $\M_{2n}$ is the minimal\footnote{Minimal here means that it is the hull of proper class of indiscernibles.} proper class mouse with $2n$ Woodin cardinals ($n\geq 1$). We say ``$\M_{2n}$ exists" if $\M_{2n}$ exists as a class and is $\omega_1+1$-iterable. 

\begin{theorem}[Woodin, \cite{PWIM}]\label{capturing projective truth} Suppose $\M_{2n}$ exists. Then for any $\Sigma^1_{2n+2}$ sentence $\phi$,
\begin{center}
$\phi \iff \M\models \phi$
\end{center}
\end{theorem}

\rthm{capturing projective truth} implies that the projective truth can be \textit{captured} by mice. A modification of its proof also implies that every real which is $\Delta^1_{2n+2}$ in a countable ordinal is in $\M_{2n}$ (the reverse direction, due to Steel, is also true, see \cite{PWIM}). The unavoidable question is how much closer can mice get to the universe or that what could the ultimate theorem on capturing definability by mice be? This question eventually leads to the \textit{Mouse Set Conjecture (MSC)},  which essentially conjectures that mice get as close to the universe as possible. However, at the moment we are not equipped with the necessary terminology to introduce MSC and we postpone our discussion of MSC to later sections. We haven't said a word on how to construct mice. 

\subsection{Constructions producing mice}

The main method for constructing mice is via \textit{background constructions}. These are constructions that use extenders of $V$ to design a coherent sequence of extenders. The extenders used in such a construction may or may not be total, i.e., measure all the subsets of their critical points. A \textit{full background construction} is a construction in which all extenders used are total. There are also $K^c$ constructions in which the extenders that are used to build the sequence may not be total. The full background constructions are more useful when the universe has large cardinals and otherwise, $K^c$ constructions are more useful. Here we will only deal with the full background constructions. 

The full background construction introduced in \cite{FSIT} by Mitchell and Steel is a construction that produces a sequence of premice $\la \M_\xi, \N_\xi : \xi \in Ord\ra$ which are the approximations of the eventual $L[\vec{E}]$. The following is only a rough sketch of the actual construction. Suppose we have defined $\la \M_\xi, \N_\xi : \xi< \a+1 \ra$. Suppose $\N_\a=L_\nu[\vec{E}]$ and there is a $(\k, \l)$-extender $F$ such that $\k<\a\leq \l$, $V_\l\subseteq Ult(V, F)$ and
 \begin{center}
 $L_\nu[\vec{E}]=L_\nu[j_F(\vec{E})\rest \nu]$.
 \end{center}
 Then our $\nu$th extender is $F^*=F\cap L_\a[\vec{E}\rest \a]$. We then let $\N_{\a+1}=(\N_\a, F^*, \in)$. If there is no such extender then we let $\nu$th member of the eventual $\vec{E}$ be the $\emptyset$ and set $\N_{\a+1}=L_\omega[\N_\a]$. 
 
To avoid coding unwanted information into $\N_{\a+1}$ we need to \textit{core down} or take a Skolem hull of $\N_{\a+1}$. $\M_{\a+1}$ is this Skolem hull. Let $\eta$ be the least cardinal of $\N_{\a+1}$ such that there is some $A\not \in \N_{\a+1}$ which is definable over $\N_{\a+1}$ from ordinals $<\eta$ and a finite sequence of ordinals $p^*$. Let $p$ be the least such sequence.  Suppose $A$ is definable via a $\Sigma_n$ formula. Then $\M_{\a+1}$ is the $\Sigma_n$-Skolem hull of $\N_{\a+1}$ where we use parameters from $\eta\cup \{p\}$. However, at this point it is not clear that there are Skolem functions for the structures we are interested in. The resolution of this problem is rather complicated and has to do with deep \textit{fine structural} facts. It is known that all of these fine structural facts follow from countable iterability. Hence, to establish these facts, we need to show that the models constructed via full background constructions are countably iterable. 
 
The eventual model is a certain limsup of $\la \M_\xi : \xi <Ord\ra$. Because $\M_\xi$ is defined to be a Skolem hull of a certain structure, it could just as well be the case that $\M_\xi=\M_\nu$ for class many $\xi$ and $\nu$, i.e., there is a stage at which we keep repeating the construction. Again, countable iterability of $\N_\xi$ implies that such repetitions cannot happen. It then follows that for each $\k$, $(\vec{E}\rest \k)^{\M_\xi}$ eventually stabilizes and converges to some $\vec{E}$. 

Mitchell and Steel showed, in \cite{FSIT}, that the countable iterability of $\M_\xi$ and $\N_\xi$ follows from countable iterability of $V$. It is then conjectured that countable iterability of $V$ is true.\\

\textbf{The Iterability Conjecture.} Suppose $\xi$ is some ordinal such that $V_\xi\models ZFC-Replacement$. Suppose $\pi: M\rightarrow V_\xi$ is elementary and $M$ is countable. Then $M$ is $\omega_1+1$-iterable.  \\

Using the arguments of Mitchell and Steel from \cite{FSIT}, it can be shown that if the full background construction converges to a model $L[\vec{E}]$ then $L[\vec{E}]$ inherits many of the large cardinals of $V$. In particular, if there is a supercompact cardinal in $V$ then there is a superstrong cardinal in $L[\vec{E}]$.

In the prelude of this paper we mentioned that the constructions producing mice are desired to be universal. The background constructions introduced above are applicable in various situations but the full background constructions will not produce mice with large cardinals unless the universe already has large cardinals. Hence, the full background constructions are not universal in the sense that they do not always produce models with large cardinals in universes which are complicated but do not have large cardinals. This is a consequence of using extenders that are total. 

$K^c$ constructions, unlike full background constructions, are indeed universal (for instance, see \rthm{JSSS}). However, the corresponding iterability conjecture for $K^c$ is much harder to prove. The advantage of descriptive inner model theoretic approach to the inner model problem is that while it still uses full background constructions to produce mice, these constructions are carried out in local iterable universes where the Iterability Conjecture is true. See for instance \rsec{partial results on imp} and \rrem{dimt approach to imp}.

 It is a well-known theorem that just in $ZFC$ there are trees of height $\omega_1$ with no branch implying that it must be very difficult to construct $\omega_1+1$-iteration strategies. It must then be the case that in order to prove the iterability conjecture we need to construct iteration strategies that are somehow nice: most likely their countable fragments have to determine the uncountable branches. The universally Baire iteration strategies have such properties. 

\subsection{Universally Baire sets and iteration strategies}

A set of reals $A$ is called $\k$-universally Baire if for any compact Hausdorf space $X$ and for any continuous function $f: X\rightarrow \mathbb{R}$, $f^{-1}[A]$ has the Baire property. A set of reals is universally Baire if it is $\k$-universally Baire for all $\k$. The universally Baire sets have canonical interpretations in all generic extension. For instance, it follows from the Shoenfield's absoluteness theorem that if $U$ is the universal $\Pi^1_1$ set of reals, $g$ is generic for some poset, $W$ is the universal $\Pi^1_1$ set of reals in $V[g]$ then $U=\mathbb{R}^V\cap W$. The universally Baire sets exhibit similar properties. First, there is an equivalent formulation. 
 
 \begin{theorem}[Feng-Magidor-Woodin, \cite{UBS}]\label{tree def of ub} A set of reals $A$ is $\kappa$-universally Baire iff 
 there are trees $T$ and $S$ on $\omega\times \kappa^+$ such that $A=p[T]$ and for any poset $\mathbb{P}$ of size $\leq \k$ and for any generic $g\subseteq \mathbb{P}$,
\begin{center}
$V[g]\models p[T]=(p[S])^c$.
\end{center}
 \end{theorem} 
 
If $(T, S)$ is as in \rthm{tree def of ub} then we say $(T, S)$ is a \textit{pair of $\k$-complementing trees}. We say $(T, S)$ is a pair of $<\k$-complementing trees if $(T, S)$ is a pair of $\l$-complementing trees for all $\l<\k$. 

Suppose $A$ is a universally Baire set and $g$ is generic for some poset $\mathbb{P}$. Let $\k>\card{\mathbb{P}}$ and let $(T, S)$ be a pair of $\k$-complementing trees witnessing that $A$ is $\k$-universally Baire. We then let $A_g$, the interpretation of $A$ on $V[g]$, be the set
\begin{center}
$A_g=(p[T])^{V[g]}$.
\end{center}
A simple absoluteness argument shows that $A_g$ is independent of $(\k, T, S)$.

\begin{lemma} $A_g$ is independent of $(\k, T, S)$.
\end{lemma}
\begin{proof}
Let $(\l, U, W)$ be another triple giving a $\l$-universally Baire representation of $A$ such that $\l>\card{\mathbb{P}}$. Let $B=p[T]^{V[g]}$ and $C=p[U]^{V[g]}$. We need to see that $B=C$. Suppose not. Then in $V[g]$, $p[S]\cap p[U]\not =\emptyset$. By absoluteness (see \rlem{capturing truth in inner models}), we have that $p[S]\cap p[U]\not =\emptyset$, contradiction. 
\end{proof}

Finally, the following lemma shows that in certain situations if our iteration strategy is such that its countable fragments are coded by a universally Baire set then it uniquely determines its uncountable branches. To state the theorem we need to explain the exact coding of $\omega_1$-strategies with sets of reals. The most straightforward way of coding an $\omega_1$-iteration strategy $\Sigma$ via a set of reals is by letting $Code(\Sigma)$ code the function $\Sigma(\T)=b$. Unfortunately, this coding isn't absolute enough but a minor modification of it is.

First, letting $\la \cdot , \cdot \ra$ be a pairing function, given $x\subseteq \omega$ we let $M_x=(\omega, E_x)$ where $E_x=\{ (n, m) : \la n, m\ra\in x\}$. If $M_x$ is well-founded then we let $x$ code the transitive collapse of $M_x$ which we denote by $N_x$. Let $\pi_x: M_x\rightarrow N_x$ be the transitive collapse. We say $x$ is a code if $M_x$ is well-founded. We then let $Code(\Sigma)$ be the set of triples $(x, n, m)$ such that $x\subseteq\omega$ is a code, $n, m\in \omega$, $\pi_x(n)$ is an iteration tree $\T$ of limit length and according to $\Sigma$, and $\pi_x(m)=\a$ such that if $b=\Sigma(\T_x)$ then $\a\in b$.  We then have that $\a\in \Sigma(\T)=b$ iff there is $(x, n, m)\in Code(\Sigma)$ is such that $\pi_x(n)=\T$ and $\pi_x(m)=\a$.

\begin{lemma}\label{universally baire strategies} Suppose $\M$ is a countable $\omega_1$-iterable mouse with an $\omega_1$-iteration strategy $\Sigma$ such that $Code(\Sigma)$ is universally Baire. Then $\Sigma$ can be extended to an $\omega_1+1$-strategy. 
\end{lemma}
\begin{proof}
We denote by $\Sigma^+$ the extension of $\Sigma$ that we are about to define. Suppose $\T$ is a length $\omega_1$-iteration tree according to $\Sigma$. Given $g\subseteq Coll(\omega, \omega_1)$, we let $\Sigma_g$ be the interpretation of $\Sigma$ onto $V[g]$. Suppose for a moment that $\Sigma_g(\T)$ is defined. Then notice that if $b=\Sigma_g(\T)$ then $b=\Sigma_h(\T)$ for any $h\subseteq Coll(\omega, \omega_1)$ such that $h$ is $V[g]$-generic. It follows then that $b\in V[g]\cap V[h]$ and hence, $b\in V$. We can then define $\Sigma^+(\T)=b$.

To finish the proof, it is enough to show that $\Sigma_g(\T)$ is defined. Suppose not. For this, let $T, S$ be trees on $\omega\times \omega_2$ witnessing that $Code(\Sigma)$ is $\omega_2$-universally Baire. Let $p[T]=Code(\Sigma)$. Let then $\pi: M\rightarrow V_\xi$ be an elementary embedding such that $M$ is countable, $\xi>\omega_2$ and $\{ T, S, \M, \T, \Sigma \} \subseteq rng(\pi)$. Let $\la \bar{T}, \bar{S},  \bar{\M}, \bar{\T}, \bar{\Sigma}\ra$ be the preimage of $\la T, S, \M, \T , \Sigma\ra$ under $\pi$. Let $\k=\omega_1^M$ and let $g\subseteq Coll(\omega, \k)$ be $M$-generic. Let $b=\Sigma(\bar{\T})$ and let $x\subseteq \omega$ and $n\in \omega$ be such that $x$ is a code and $\pi_x(n)=\T$. It then follows that for any $m\in \omega$, 
\begin{center}
$\pi_x(m)\in b$ iff $(x, n, m)\in (p[\bar{T}])^M$.
\end{center}
 To see this, notice that if $m\in \omega$ is such that $\pi_x(m)\in b$ but $(x, n, m)\not \in (p[\bar{T}])^M$ then $(x, n, m)\in (p[\bar{S}])^M$. Applying $\pi$ we get that $(x, n, m) \in p[S]$ and hence, the node coded by $m$ is not in $b$. It follows from the above equivalence that $b\in M[g]$ and therefore, $\bar{\Sigma}_g(\bar{\T})$ is defined. This is a contradiction, as by elementarity of $\pi$, exactly the opposite must hold. 
\end{proof}

One way or another, the construction of iteration strategies whose $\omega_1$-parts determine their $\omega_1+1$ part seems to be unavoidable. Often times such strategies come in the form of universally Baire sets. As complicated mice have complicated iteration strategies, our constructions producing universally Baire iteration strategies start producing more complicated universally Baire sets. The most natural environment to study the complexity of universally Baire sets is given by axioms of the Solovay hierarchy, which is the topic of the next section.

\section{The Solovay hierarchy}

In this section our goal is to introduce the Solovay hierarchy which is a hierarchy of axioms that provide rich environment to study universally Baire sets. The analysis of the $\H$'s of models of the Solovay hierarchy has been one of the central projects of descriptive inner model theory. This analysis is the subject of the next section and here, our goal is to develop the background material. The base theory in the Solovay hierarchy is $AD^+$, an extension of $AD$ introduced by Woodin. We start by introducing it. 

\subsection{$AD^+$}

 Most of the descriptive set theoretic notions that we will need can be found in \cite{Moschovakis}. Following the tradition of descriptive set theorists, we let the reals be the members of the Baire space $\omega^\omega$. 

Recall that $AD_X$ is the axiom asserting that all two player games of perfect information of length $\omega$ where players play members of $X$ are determined. More precisely, given a set  $A\subseteq X^\omega$, consider the two player game $\mathcal{G}^X_A$ in which players $I$ and $II$ take turns to play $x_0, x_1,..., x_{2i}, x_{2i+1},...$ and $I$ wins if the sequence $\la x_0, x_1,..., x_k,...\ra\in A$. $AD_X$ then says that for every $A\subseteq X^\omega$, $\mathcal{G}^X_A$ is determined, i.e., one of the players has a winning strategy.

$AD$ is $AD_\omega$ and $AD_{\mathbb{R}}$ is the axiom $AD_X$ where $X=\mathbb{R}$.
It is well known that $AD$ contradicts $AC$. Solovay showed that $AD_{\mathbb{R}}$ is consistencywise much stronger than $AD$ (see \cite{Solovay}). Also, unlike $AD_{\mathbb{R}}$ which fails in $L(\mathbb{R})$, if $AD$ holds in any inner model of $ZF$ containing the reals then $AD$ holds in $L(\mathbb{R})$.

The work presented in \cite{Jackson},  \cite{ExtentScales}, \cite{Moschovakis}, and \cite{Scales} shows that under $AD$, $L(\mathbb{R})$ has a very canonical structure and can be analyzed in detail. Woodin defined $AD^{+}$, a strengthening of $AD$, which yields a similar analysis of arbitrary models of $AD^++V=L(\powerset(\mathbb{R}))$. $AD^+$ is essentially the theory one gets when every set of reals is \textit{Suslin} in some bigger model of $AD$. In particular, if $M\subseteq N$ are two transitive models of $AD$ containing $\mathbb{R}$ such that every set of reals of $M$ is Suslin in $N$, then $M\models AD^+$.  Woodin conjectured that $AD$ implies $AD^+$ and showed that $AD_{\mathbb{R}}+DC$ implies $AD^+$. Because of this, the readers not familiar with $AD^+$ may ignore the $+$ without losing much.

Given a cardinal $\k$ and a set of reals $A$, we say $A$ is $\k$-Suslin if there is a tree $T$ on $\omega\times \k$ with the property that $p[T]=A$. We say $T$ is a \textit{$\k$-Suslin representation} of $A$. $\kappa$ is called a Suslin cardinal if there is a $\k$-Suslin set $A\subseteq \mathbb{R}$ which is not $\l$-Suslin for all $\l<\k$. It is not hard to show that the collection of $\k$-Suslin sets is closed under projections.  AC  implies that all sets of reals are Suslin and the notion is only interesting when we require that the Suslin representation be definable in some natural way. Recall, for instance, that it follows from the proof of Shoenfield's absoluteness theorem that every $\Pi^1_1$ set has a Suslin representation which is constructible. Also, over $V=L(\mathbb{R})$, $AC$ is equivalent to the statement ``all sets of reals are Suslin". Hence, under $AD$, there are sets of reals in $L(\mathbb{R})$ which do not have Suslin representation in $L(\mathbb{R})$.  Martin and Woodin showed that in the presence of $AD$, $AD_{\mathbb{R}}$ is equivalent to all sets of reals being Suslin.

To introduce $AD^+$ we need two notions, \textit{$^\infty$Borel sets} and \textit{ordinal determinacy}. A set of reals $A$ is called $^\infty$Borel if there is a set of ordinals $S$, an ordinal $\a$ and a formula $\phi(x_0, x_1)$ such that

 \begin{center}
 $x\in A \leftrightarrow L_\alpha[S, x]\models \phi[S, x]$.
 \end{center}
$(S, \alpha, \phi)$ is called an $^\infty$Borel code for $A$. One can give an equivalent definition using infinitary languages in which case the resemblance to ordinary Borel sets becomes apparent. Note that if $A$ is Suslin and $T$ is such that $p[T]=A$ then $T$ witnesses that $A$ is $^\infty$Borel. It follows that all Suslin sets are $^\infty$Borel.

Ordinal determinacy is the following statement.\\

\textbf{Ordinal Determinacy:}\index{ordinal determinacy} For any $\l<\Theta$, for any continuous function $\pi:\l^\omega \rightarrow \omega^\omega$, and for any set $A\subseteq \omega^\omega$ the set $\pi^{-1}[A]$ is determined.\\ 

Below, $DC_{\mathbb{R}}$ stands for the \textit{Axiom of Dependent Choice} for the relations on $\mathbb{R}$. 

\begin{definition}[Woodin]\label{ad+} $ AD^+$ is the following theory
\begin{enumerate}
\item ZF+AD+DC$_{\mathbb{R}}$.
\item All sets of reals are $^\infty$Borel.
\item Ordinal determinacy.
\end{enumerate}
\end{definition}

Let $\Gamma^s=\{ A\subseteq \mathbb{R}: A$ and $A^c$ are Suslin$\}$. It can be shown that $AD^++V=L(\powerset(\mathbb{R}))$ is equivalent to $AD+V=L(\powerset(\mathbb{R}))+L(\Gamma, \mathbb{R})\prec_{\Sigma_1} V$.
$AD^+$ provides a rich environment to study canonical sets of reals such as universally Baire sets. However, $AD^+$ by itself doesn't give us a way of measuring the complexity of its models. For this we need a hierarchy of axioms extending $AD^+$. The hierarchy we will use is the \textit{Solovay hierarchy}.

\subsection{The Solovay hierarchy}\label{solovay hierarchy}

Complicated mice have complicated iteration strategies. Since mice are complicated because of the kind of large cardinals they have, to be able to construct iteration strategies for complicated mice in models of determinacy one needs to have a hierarchy of axioms for the models of determinacy covering the entire spectrum of the consistency strength hierarchy given by large cardinals. We believe that the \textit{Solovay hierarchy} is one such hierarchy (see \rprob{main problem} and \rcon{dimt conjecture}). The base theory in the Solovay hierarchy is $AD^+$.

First, recall the \textit{Wadge ordering} of $\powerset(\mathbb{R})$. For $A, B\subseteq \mathbb{R}$, we say $A$ is \textit{Wadge reducible} to $B$ and write $A\leq_W B$ if there is a continuous function $f:\mathbb{R} \rightarrow \mathbb{R}$ such that $x\in A \iff f(x)\in B$. Martin showed that under $AD$, $\leq_W$ is well founded. Also, Wadge showed that under $AD$, for any two sets of reals $A$ and $B$, either $A\leq_W B$ or $B\leq_W A^c$. For $A\subseteq \mathbb{R}$, we let $w(A)$ be the rank of $A$ in $\leq_W$. It also follows that $\sup_{A\subseteq \mathbb{R}}w(A)=\Theta$. 

To define the Solovay hierarchy, we need to define the \textit{Solovay sequence} which is a closed increasing sequence $\la \theta_\a : \a\leq \Omega\ra$ of ordinals defined by
\begin{enumerate}
\item $\theta_0=\sup\{ \a :$ there is a surjection $f: \mathbb{R}\rightarrow \a$ such that $f$ is $OD\}$,
\item if $\theta_{\b}<\Theta$ then 
    \begin{center}
    $\theta_{\b+1}=\sup \{ \a :$ there is a surjection $f: \mathbb{R}\rightarrow \a$ such that $f$ is $OD_A\}$,
    \end{center}
    where $A\subseteq \mathbb{R}$ is such that $w(A)=\theta_\b$,
\item if $\l$ is a limit then $\theta_\l=\sup_{\a<\l}\theta_\a$.
\end{enumerate}

The Solovay hierarchy is the hierarchy we get by requiring that $\Omega$ is large. The following are the first few theories of this hierarchy. $\leq_{con}$ is the consistency strength relation: given two theories $T$ and $S$, $S\leq_{con} T$ if $Con(T)\vdash Con(S)$.
\begin{center}
$AD^{+}+\Omega=0<_{con} AD^{+}+\Omega=1 < AD^{+}+\Omega=2 \cdot \cdot \cdot <_{con} AD^{+}+\Omega=\omega <_{con}\cdot \cdot \cdot <_{con} AD^{+}+\Omega=\omega_1 <_{con} AD^{+}+\Omega= \omega_1+1 \cdot \cdot \cdot$.
\end{center}

Our presentation of the Solovay hierarchy depends on the ordinal parameter $\Omega$ and when $\Omega$ itself isn't definable, we get a theory in a parameter. This is an inconvenience and in particular, it suggests that perhaps the Solovay hierarchy can never cover the entire consistency strength hierarchy given by the large cardinal axioms. However, we conjecture that the Solovay hierarchy is indeed a consistency strength hierarchy and we will give the details in the next subsection. The following is a list of basic facts on the Solovay sequence. 

For $\a\leq \Omega$, we let
\begin{center}
$\Gamma_\a=\{ A\subseteq \mathbb{R} : w(A)<\theta_\a\}$.
\end{center}
The $\Gamma_\a$'s are called \textit{Solovay pointclasses}. The following lemma isn't hard to prove and we leave it as an exercise.

\begin{lemma}\label{solovay pointclasses} The following holds.
\begin{enumerate}
\item If $A\subseteq \mathbb{R}$ and $L(A, \mathbb{R})\models AD^+$ then for some $\a$, $L(A, \mathbb{R})\models \Theta=\theta_{\a+1}$.
\item Assume $AD^+$ and let $\la \theta_\a :\a\leq \Omega\ra$ be the Solovay sequence. Then for every $\a$, $L(\Gamma_\a, \mathbb{R})\models \theta_\a=\Theta$.
\end{enumerate}

\end{lemma}

Woodin showed that the $\theta_\a$'s are Suslin cardinals. While we are not aware of published account of the following theorem, \cite{Jackson} contains many results that are used to proof it.

\begin{theorem}\label{on thetas} Assume $AD^{+}+V=L(\powerset(\mathbb{R}))$ and let $\la \theta_\a :\a\leq \Omega\ra$ be the Solovay sequence. Then the following holds.
 \begin{enumerate}
 \item (Martin-Woodin) For every $\a<\Omega$, $\theta_\a$ is a Suslin cardinal.
 \item (Kechris-Solovay) For every $\a\geq -1$, if $\a+1\leq \Omega$ then $\theta_{\a+1}$ isn't a limit of Suslin cardinals.
 \item (Martin-Woodin) For every $\a\geq -1$, if $\a+1<\Omega$ then $\cf(\theta_{\a+1})=\omega$.
 \end{enumerate}
\end{theorem}

Woodin also showed that over $AD^{+}+V=L(\powerset(\mathbb{R}))$, $AD_{\mathbb{R}}$ is equivalent to an axiom from the Solovay hierarchy. The following follows from clause 1 and clause 2 of \rthm{on thetas} and the fact, due to Martin and Woodin, that $AD_{\mathbb{R}}$ is equivalent to $AD+``$ all sets of reals are Suslin".

\begin{corollary}\label{adr and solovay} Assume $AD^{+}+V=L(\powerset(\mathbb{R}))$. Then the following theories are equivalent.
\begin{enumerate}
\item $AD_{\mathbb{R}}$.
\item $\Omega$ is a limit ordinal.
\end{enumerate}
\end{corollary}

The following theorem connects the models of the Solovay hierarchy to large cardinals.  Recall that $\H$ is the class of \textit{hereditarily ordinal definable sets}. More precisely, $X\in \H$ iff every member of $Tc(\{X\})$ is ordinal definable. 

\begin{theorem}[Woodin, \cite{KoelWoodin}] \label{woodins in hod} Assume $AD$. For $\a\geq -1$, if $\theta_{\a+1}$ is defined then 
\begin{center}
$\H\models \theta_{\a+1}$ is a Woodin cardinal.
\end{center}
\end{theorem}

\rthm{woodins in hod} seems to suggest that perhaps the large cardinal structure of $\H$ of the models of the Solovay hierarchy is a good way of measuring the large cardinal strength of the Solovay hierarchy. In particular, notice that as the theories from the Solovay hierarchy become stronger the number of Woodin cardinals in $\H$ grows. A more careful analysis of $\H$ suggests that this view is indeed correct and we will explain the details of this point of view in \rsec{hod analysis}. 

\subsection{Some axioms from the Solovay hierarchy}

The following is a list of natural closure points of the Solovay hierarchy listed according to their consistency strengths starting from the weakest. 

\begin{enumerate}
\item $AD_{\mathbb{R}}+``\Theta$ is regular".
\item $AD_{\mathbb{R}}+``\Theta$ is Mahlo in $\H$".
\item $AD_{\mathbb{R}}+``\Theta$ is weakly compact in $\H$".
\item $AD_{\mathbb{R}}+``\Theta$ is measurable".
\item $AD_{\mathbb{R}}+``\Theta$ is Mahlo"
\end{enumerate} 

\begin{lemma} The above set of axioms are  listed according to their consistency strength starting from the weakest.
\end{lemma}
\begin{proof}
The main ingredient of the proof is the following sublemma. \\

\textit{Sublemma.} Assume $AD^{+}+V=L(\powerset(\mathbb{R}))$. Then $\Theta$ is regular iff $\H\models ``\Theta$ is regular".
\begin{proof} We only give the idea. The case $\Theta=\theta_{\a+1}$ is similar to the proof of the same fact in $L(\mathbb{R})$. Assume $\Theta=\theta_\a$ where $\a$ is limit. Suppose $\Theta$ is singular yet $\H\models ``\Theta$ is regular". Fix a singularizing function $f: \l \rightarrow \Theta$. We have that every set is ordinal definable from a set of reals. Let then $A\subseteq \mathbb{R}$ be such that $f$ is $OD$ from $A$. It also follows from $AD^+$ that every set of reals is ordinal definable from a bounded subset of $\Theta$, namely its $\infty$-Borel code. Let then $B\subseteq \b< \Theta$ be such that $A$ is $OD$ from $B$. Using Vop\v{e}nka algebra, we get that $B$ is generic over $\H$. Let $G_B$ be the generic object for the Vop\v{e}nka algebra such that $B\in \H[G_B]$. We have that $\H[G_B]=\H_{G_B}$. The key point is that the Vop\v{e}nka algebra that adds $B$ to $\H$ has size at most $\theta_{\gg+2}$ where $\gg$ is least such that $\b<\theta_{\gg+2}$. We now have that $f\in \H[G_B]=\H_{G_B}$ and therefore, $\Theta$ is singular in $\H[G_B]$. This is a contradiction because forcing of size $\theta_{\gg+2}$ cannot singularize $\Theta$.
\end{proof}

Our sublemma immediately implies that Axiom 2 is at least as strong as Axiom 1. To see that it is actually stronger than Axiom 1, fix $M$ which is a model of $AD^++``\Theta$ is Mahlo in $\H$". Work in $M$. It then follows that there must be $\b$ such that $\theta_\b$ is regular in $\H$ and $\theta_\b<\Theta$. Using \rlem{solovay pointclasses}, \rcor{adr and solovay} and our sublemma we get that  $L(\Gamma_\b, \mathbb{R})\models AD_{\mathbb{R}}+``\Theta$ is regular". This shows that Axiom 1 is strictly weaker than Axiom 2. 

The rest of the implications are similar. Particularly interesting is the fact that Axiom 5 is stronger than Axiom 4. This is because building on an earlier work of Steel, Woodin showed that under $AD^{+}$ every regular cardinal $<\Theta$ is measurable. Since the Solovay sequence is a club, it follows that if $\Theta$ is Mahlo then there are stationary many members of the Solovay hierarchy that are measurable cardinals. Fix then one such $\theta_\b<\Theta$. Let $\mu$ be a normal $\theta_\b$-complete measure on $\theta_\b$. By Kunen's result, $\mu$ is ordinal definable. Hence, $\powerset(\mathbb{R})\cap L(\Gamma_\b, \mu)=\Gamma$. It then follows that $L(\Gamma_\b, \mu)\models AD_{\mathbb{R}}+``\Theta$ is measurable". \end{proof}

The next important axiom of the Solovay hierarchy is the axiom known as $LST$ which stands for \textit{``the largest Suslin cardinal is a theta"}. It can be stated as follows.\\

\textbf{LST:}  $AD^{+}+V=L(\powerset(\mathbb{R}))+\Omega=\a+1+``\theta_\a$ is the largest Suslin cardinal".\\

 As the next proposition shows, $LST$ is a stronger theory than $AD_{\mathbb{R}}+``\Theta$ is regular".

\begin{proposition}\label{lst implies adr+theta is reg}
Assume $LST$. Then there is $\Gamma\varsubsetneq\powerset(\mathbb{R})$ such that $L(\Gamma, \mathbb{R})\models AD_{\mathbb{R}}+``\Theta$ is regular".
\end{proposition}
\begin{proof}
Let $\a$ be such that $\a+1=\Omega$. Thus, $\theta_\a$ is the largest Suslin cardinal. It is well-known theorem that the largest Suslin cardinal, if exists, is a regular cardinal and is a limit of Suslin cardinals. Hence, $\theta_\a$ is a regular cardinal. Moreover, $\a$ must be a limit ordinal as no $\theta_{\xi+1}$ is a limit of Suslin cardinals (see 2 of \rlem{on thetas}). It then follows that $L(\Gamma_\a, \mathbb{R})\models AD_{\mathbb{R}}+``\Theta$ is regular". In fact, by 3 of \rlem{on thetas}, $\Theta^{L(\Gamma_\a, \mathbb{R})}=\theta_\a$.
\end{proof}

A similar argument shows that $LST$ is stronger than Axioms 1-4 in our first list of axioms. A result of Kechris, Klienberg, Moschovakis, and Woodin, proved in \cite{KKMW}, implies that it is stronger than Axiom 5. As was mentioned before, $AD_{\mathbb{R}}+``\Theta$ is regular" and hence, $LST$ were considered to be quite strong, as strong as supercompact cardinals. However, 4 of the Main Theorem shows that $AD_{\mathbb{R}}+``\Theta$ is regular" is much weaker than supercompact cardinals.  \rthm{theta regular from large cardinals} gives an equiconsistency for $AD_{\mathbb{R}}+``\Theta$ is regular".  The exact large cardinal strength of Axioms 2-5 is unknown. As for $LST$, we conjecture that it is in the region of Woodin limit of Woodin cardinals. \rsec{speculations} has more on this topic.

The axioms we have listed so far do not cover the entire spectrum of the consistency strength hierarchy given by large cardinals. In fact, it is not at all clear where the next list of axioms should come from. While we do have an educated guess as to where these axioms should come from, we leave the door open for other possible extensions of $AD^+$. This is the motivation behind stating the following problem in vague terms.
 
\begin{problem}\label{main problem} Find extensions of $AD^+$ that cover the entire spectrum of the large cardinal hierarchy. 
\end{problem} 

Below we take the position that by calibrating the length of the Solovay sequence we will cover the entire spectrum of the large cardinal hierarchy. The study of $\H$ of models of the Solovay hierarchy suggests that one way the Solovay sequence gets longer is when $\H$ starts having larger and larger cardinals in it. This phenomenon suggests that the large cardinal structure of $\H$ is the correct measure of the complexity of the models of the Solovay hierarchy. An evidence of this was already present in \rthm{woodins in hod}. 
 
The following axioms have been motivated by this view and we conjecture that they are consistent relative to large cardinals.
 Given a large cardinal axiom $\psi:=\exists \kappa \phi(\kappa)$ let
 \begin{center}
 $S_\phi=_{def}``AD^{+}+V_\Theta^\H\models \psi$''.
 \end{center}
 We then make the following conjecture.
 
 \begin{conjecture}\label{dimt conjecture} For each large cardinal axiom $\phi$, $S_\phi$ is consistent relative to some large cardinal axiom. In particular, letting $\phi=``$there is a supercompact cardinal", $S_\phi$ is consistent relative to some large cardinal axiom.
 \end{conjecture}

\textbf{A confession:} I believe that \rcon{dimt conjecture} is the most important open problem of descriptive inner model theory. In particular, I believe that it cannot have a ``cheap" solution. Its resolution, either way, will be based on deep facts exploiting both the machinery from inner model theory and from descriptive set theory. In the unfortunate case that it is resolved negatively, its negative resolution ought to come with an alternative approach to the inner model problem. 
 
The following two theorems give an alternative confirmation that there must be theories as in \rprob{main problem}. The conclusion of both of these theorems can be forced from large cardinal axioms in the region of supercompact cardinals. The author, however, showed that the hypothesis of both theorems is weaker than a Woodin limit of Woodin cardinals (for the first theorem this follows from part 4 of the Main Theorem). It is known that both conclusions have a significant large cardinal strength. We have already mentioned the first theorem, but, for the convenience of our readers, we will record it here again.

\begin{theorem}[Woodin] Assume $AD_{\mathbb{R}}+``\Theta$ is regular". Then there are posets $\mathbb{P}$ and $\mathbb{Q}$ such that 
\begin{enumerate}
\item $V^{\mathbb{P}}\models MM(c)$
\item $V^{\mathbb{Q}}\models CH$ +``there is an $\omega_1$-dense ideal on $\omega_1$".
\end{enumerate}
\end{theorem}

\begin{theorem}[Caicedo, Larson, S., Schindler, Steel, Zeman, \cite{CLSSSZ}] Assume $AD_{\mathbb{R}}+``\Theta$ is regular". Suppose the set $\{ \k<\Theta : \k$ is regular in $\H$ and $\cf(\k)=\omega_1\}$ is stationary in $\Theta$. Then there is $\mathbb{P}$ such that
\begin{center}
$V^{\mathbb{P}}\models \neg \square(\omega_2)+\neg\square_{\omega_2}$.
\end{center}
\end{theorem}

Both theorems suggest a more formal version of \rprob{main problem}. 

\begin{problem}\label{formal main problem} Force PFA over models of $AD^+$. Force the statement $``\forall \k>\omega_1(\neg \square_{\kappa})"$ over models of $AD^+$. 
\end{problem}

It is a well known theorem of Todorcevic that PFA implies that $``\forall \k>\omega_1(\neg \square_{\kappa})"$. Thus, the second problem is an easier version of the first. The following theorem suggest that a positive solution to the second part of \rprob{formal main problem} for $\k=\omega_3$ will most likely yield a solution to the inner model problem for superstrong cardinals.

\begin{theorem}[Jensen, Schimmerling, Schindler, Steel, \cite{JSSS}]\label{JSSS} Assume $\neg \square(\omega_3)+\neg \square_{\omega_3}$ and that $\omega_2^{\omega}=\omega_2$. Moreover, suppose $K^c$ exists. Then there is a superstrong cardinal in $K^c$.
\end{theorem}
 
\subsection{On the consistency of $S_\phi$}\label{consistency of sphi}

One of the main techniques for reducing the consistency of the Solovay hierarchy to that of large cardinal hierarchy is Woodin's \textit{derived model theorem}. The theorem associates to each cardinal $\l$, which is a limit of Woodin cardinals, a canonical model of $AD^+$. 

Suppose $\lambda$ is a limit of Woodin cardinals. First let $Coll(\omega, <\l)$ be the Levy collapse of $\l$ to $\omega$ and let $G\subseteq Coll(\omega, <\l)$ be generic. For $\a<\l$, let $G_\a=G\cap Coll(\omega, <\a)$ and let
\begin{center}
$\mathbb{R}^*=\cup \{ \mathbb{R}^{V[G_\a]} : \a<\l\}$.
\end{center}
Recall that $V(\mathbb{R}^*)$ is the minimal inner model of $V[G]$ containing $V$ and $\mathbb{R}^*$ and satisfying $ZF$.
Let
\begin{center}
$\Gamma_G=\{ A\subseteq \mathbb{R}^* : A\in V(\mathbb{R}^*) \wedge L(A, \mathbb{R}^*)\models AD^+\}$.
\end{center}

\begin{theorem}[Woodin, The Derived Model Theorem]\label{derived model} Suppose $\l$ is a limit of Woodin cardinals. Then in $V(\mathbb{R}^*)$, 
\begin{center}
$\Gamma_G=\powerset(\mathbb{R}^*)\cap L(\Gamma_G, \mathbb{R}^*)$ and $L(\Gamma_G, \mathbb{R}^*)\models AD^+$.
\end{center}
\end{theorem}

The model $L(\Gamma_G, \mathbb{R})$ is called \textit{the derived model} associated to $\l$. While the derived model depends on $G$, its first order theory is independent of $G$ as the model is definable in a homogeneous forcing extension. The proof of \rthm{derived model} is unpublished but see \cite{StationaryTower}, \cite{DMT}, \cite{Steel2007} and \cite{DMATM} for published proofs of various weak versions of \rthm{derived model}. \cite{ZHUHODNotes} contains the full proof.

The following theorem establishes the consistency of some theories from the Solovay hierarchy. We say $\l$ \textit{satisfies $AD_{\mathbb{R}}$-hypothesis} if it is a limit of Woodin cardinals and $<\l$-strong cardinals. We let $AD_{\mathbb{R}}$-hypothesis be the statement: there is $\l$ satisfying $AD_{\mathbb{R}}$-hypothesis. 

\begin{theorem}[Woodin, \cite{DMT}] Suppose $\l$ satisfies the $AD_{\mathbb{R}}$-hypothesis. Then the following holds.
\begin{enumerate}
\item The derived model at $\l$ satisfies $AD_{\mathbb{R}}$.
\item If $\l$ is a regular cardinal then the derived model at $\l$ satisfies $\Theta=\theta_\Omega$ for some $\Omega$ of uncountable cofinality. 
\end{enumerate}
\end{theorem}

Recently Steel, building on an earlier work of Woodin, showed that $AD_{\mathbb{R}}$ implies that there is premouse satisfying $AD_{\mathbb{R}}$-hypothesis. 

\begin{theorem}[Steel-Woodin]\label{adr equiconsistency} The following theories are equiconsistent.
\begin{enumerate}
\item $ZF+AD_{\mathbb{R}}$.
\item $ZFC+AD_{\mathbb{R}}$-hypothesis.
\end{enumerate}
\end{theorem}

The following theorem establishes the consistency of the theory $AD_{\mathbb{R}}+``\Theta$ is regular". The hypothesis is weaker than Woodin limit of Woodins. We say the triple $(\k, \dot{T}, \dot{S})$ codes a $<\d$-universally Baire set if $\k<\d$, $\dot{T}, \dot{S}\in V^{Coll(\omega, \k)}$ and in $V^{Coll(\omega, \k)}$, $(\dot{T}, \dot{S})$ is a pair of $<\d$-complementing trees. Thus, if $(\k, \dot{T}, \dot{S})$ codes a $<\d$-universally Baire set then whenever $g\subseteq Coll(\omega, \k)$ is generic, in $V[g]$, $p[\dot{T}_g]$ is $<\d$-universally Baire as witnessed by $(\dot{T}_g, \dot{S}_g)$. We let $\dot{\Gamma}^\d_{ub}$ be the set of triples $(\k, \dot{T}, \dot{S})$ which code a $\d$-universally Baire set. 

\begin{definition} Suppose $\k<\d$ are two inaccessible cardinals and $A\subseteq \dot{\Gamma}^\d_{ub}$. We say $\k$ coheres $A$ if for every  $\l\in (\k, \d)$ there is an extender $E\in V_\d$ such that
\begin{enumerate}
\item $\cp(E)=\k$ and $j_E(\k)\geq \l$,
\item $V_\l\subseteq M=Ult(V, E)$,
\item for any $(\nu, \dot{T}, \dot{S})\in A$ such that $\nu<\l$, in $V^{Coll(\omega, \nu)}$, the $\l$-universally Baire set given by $(\dot{T}\rest \l , \dot{S}\rest \l)$ is a $<\d$-universally Baire set in $M^{Coll(\omega, \nu)}$.
\end{enumerate}
\end{definition}

Given cardinals $\k<\d$ and $A\subseteq \d$, we say $\k$ is $A$-reflecting if for any $\l\in (\k, \d)$, there is a $(\k, \l)$-extender $E$ such that $\cp(E)=\k$ and $j_E(A)\cap \l =A\cap\l$.

\begin{proposition}\label{easy proposition} Suppose $\d$ is an inaccessible limit of $<\d$-strong cardinals and there is $\k<\d$ such that it reflects the set of $<\d$-strong cardinals. Then $\k$ coheres $\dot{\Gamma}^\d_{ub}$.
\end{proposition}
\begin{proof}
The proof is an easy application of the fact that if a set is $\nu$-universally Baire and $\nu$ is a strong cardinal then it is universally Baire. 
\end{proof}

Suppose $\d$ is a cardinal and $\Gamma\subseteq \dot{\Gamma}^\d_{ub}$. Suppose $g\subseteq Coll(\omega, <\d)$ is generic and $b=(\k, \dot{T}, \dot{S})$ codes a $<\d$-universally Baire set. We let $A_{b, g}=\cup_{\l<\d} (p[\dot{T}_{g\cap Coll(\omega, \k)}])^{V[g\cap Coll(\omega, \l)]}$.  Let $\mathbb{R}^*$ be the symmetric reals of $V[g]$ and working in $V(\mathbb{R}^*)$, let 
\begin{center}
$\Gamma^g=\{ A_{b, g}: b\in\Gamma\}$.
\end{center}
We say $\Gamma$ is \textit{closed} if for any $g$, $L(\Gamma^g, \mathbb{R}^*)\models AD^+$ and in $V(\mathbb{R}^*)$, $\Gamma^g=\powerset(\mathbb{R}^*)\cap L(\Gamma, \mathbb{R}^*)$. Given a sentence $\phi$ and a closed $\Gamma$, write $\Gamma\models \phi$ if $L(\Gamma^g, \mathbb{R}^*)\models \phi$ for all $g$. Let $\phi_{sing}$ be the sentence $``\Theta$ is singular". 

We say $\d$ \textit{satisfies $\Theta$-regular hypothesis} if it is an inaccessible limit of Woodin cardinals and $<\d$-strong cardinals and whenever $\Gamma\subseteq \dot{\Gamma}^\d_{ub}$ is such that $\Gamma\models \phi_{sing}$ then there is $\k<\d$ such that $\k$ coheres $\Gamma$. We let $\Theta$-regular hypothesis be the statement: there is $\d$ which satisfies $\Theta$-regular hypothesis. 

\begin{theorem}\label{theta regular from large cardinals} Suppose $\d$ satisfies the $\Theta$-regular hypothesis. Let $G\subseteq Coll(\omega, <\d)$ be generic and let $\Gamma_G$ and $\mathbb{R}^*$ be as in \rthm{derived model}. Let $\la \theta_\a :\a\leq \Omega\ra$ be the Solovay sequence of $\Gamma_G$. Then, for some $\a\leq \Omega$, if $\Gamma=(\Gamma_\a)^{L(\Gamma_G, \mathbb{R}^*)}$, then
\begin{center}
$L(\Gamma, \mathbb{R}^*)\models AD_{\mathbb{R}}+``\Theta$ is regular".
\end{center}
\end{theorem}

Recently Yizheng Zhu reversed the conclusion of \rthm{theta regular from large cardinals}. 

\begin{theorem}[S.-Zhu, \cite{SargZhu}]\label{theta-reg hypo} The following theories are equiconsistent.
\begin{enumerate}
\item $ZFC+\Theta$-regular hypothesis. 
\item $ZF+AD_{\mathbb{R}}+``\Theta$ is regular".
\end{enumerate}
\end{theorem}

At this point we do not have a proof of consistency of $LST$ relative to large cardinals. It is one of the most important open problems of this area. 

\begin{problem}[The LST Problem] Establish the consistency of $LST$ relative to some large cardinal axiom. 
\end{problem}

\section{$\H$ of models of $AD^+$}\label{hod analysis} 

The problem of understanding the structure of $\H$ of models of $AD$ has been one of the central projects of descriptive set theory. Over the years, number of deep results have been proved on the structure of $\H$ under $AD$ and we take the following list of theorems as our starting point.

\begin{theorem}\label{ch in hod} Assume $AD$. Then the following holds.
\begin{enumerate}
\item (Folklore) Suppose $V=L(\mathbb{R})$. Then $\H\models CH$.
\item (Solovay, \cite{Kanamori}) $\omega_1$ is measurable in $\H$.
\item (Becker, \cite{Becker}) $\omega_1$ is the least measurable cardinal of $\H$.
\end{enumerate}
\end{theorem}

The theorem suggests that under $AD$, $\H$ has a rich structure and \rthm{woodins in hod} is a confirmation of it. Woodin's derived model theorem, \rthm{derived model}, opened up the door for further explorations and the following two theorems were proved soon after. Given a set $X$, we let $Tc(X)$ be the least transitive set containing $X$. A set $X$ is called \textit{self-wellordered} if there is a well-ordering of $Tc(\{X\})$ in $L_{\omega}(Tc(\{X\}))$. Given a self-wellordered set $X$ we say $\M$ is a mouse over $X$ if $\M$ has the form $L_\a[\vec{E}][X]$. The distinction between ``a mouse" and ``a mouse over $X$" is the same as the distinction between $L$ and $L[X]$. Notice that every real is self-wellordered.

Let $\M_\omega(y)$ be the least\footnote{Here ``least" means that it is the hull of club of indiscernibles.} class size $y$-mouse with $\omega$ Woodin cardinals. Given a real $y$, we say $``\M_\omega(y)$ exists" if $\M_\omega(y)$ exists as a class and it is $\k$-iterable for all $\k$. 

\begin{theorem}[Woodin, \cite{OIMT}] Suppose $\M_\omega$ exists. Then $AD$ holds in $L(\mathbb{R})$. 
\end{theorem}

\begin{theorem}[Steel-Woodin, \cite{OIMT}] \label{msc in l(r)} Suppose $\M_\omega$ exists. Then in $L(\mathbb{R})$, $x$ is ordinal definable iff $x$ is in some mouse. Moreover, the following statements are equivalent where $x, y\in \mathbb{R}$. 
\begin{enumerate}
\item $L(\mathbb{R})\models ``x$ is ordinal definable from $y$".
\item $x\in \M_\omega(y)$.
\end{enumerate}
\end{theorem}

\begin{corollary} \label{reals of hod of l(r)} Assume $\M_\omega$ exists. Let $\mathcal{H}=\H^{L(\mathbb{R})}$. Then 
\begin{center}
 $\mathbb{R}^\mathcal{H}=\mathbb{R}^{\M_\omega}$.
\end{center}
\end{corollary}

Since $CH$ holds in every mouse, \rcor{reals of hod of l(r)} gives another proof of 1 of \rthm{ch in hod}. Notice that \rthm{msc in l(r)} is an instance of capturing definability via mice, a topic we discussed in \rsec{levy hierarchy and mice}. \rcor{reals of hod of l(r)} also gives a nice characterization of the reals of $\H$ and it is impossible not to ask if there can be such a characterization of the rest of $\H$. 

The theorems stated in this prelude motivate number of questions having to do with  the generality of these theorems. Here is a list of such questions.
\begin{question} \label{questions} Assume $AD^++V=L(\powerset(\mathbb{R}))$.
\begin{enumerate}
\item Does $\H\models CH$?
\item Does $\H\models GCH$?
\item Is it true that the reals of $\H$ are the reals of some mouse?
\item What kind of large cardinals does $\H$ have?
\item What is the structure of $\H$?
\end{enumerate}
\end{question}

The following surprising theorem of Woodin and the proof of \rthm{ch in hod} imply that the answer to the first question is yes.

\begin{theorem}[Woodin]\label{sigma21} Assume $AD^{+}+V=L(\powerset(\mathbb{R}))$. Then the set
\begin{center}
$A=\{ (x, y) \in \mathbb{R}^2 : x\in OD(y) \}$
\end{center}
is $\Sigma^2_1$. Moreover, for some $\k$ there is a tree $T\in \H$ on $\omega\times \kappa$ such that $A=p[T]$. 
\end{theorem}

\begin{corollary} Assume $AD^{+}+V=L(\powerset(\mathbb{R}))$. Then $\H\models CH$.
\end{corollary}
\begin{proof} Let $A=\{ x \in \mathbb{R} : x\in OD \}$. Fir $x\in A$, let $(\phi_x, \vec{\a}_x)$ be the lexicographically ($\leq_{lex}$) least such that $x$ is definable from $\vec{\a}_x$ via $\phi_x$. Let then $\leq^*$ be the $OD$ wellordering of $A$ given by $y\leq^* x$ iff $(\phi_y, \a_y)\leq_{lex} (\phi_x, \a_x)$. It follows from \rthm{sigma21} that there is $T\in \H$ such that $p[T]=\leq^*$. But then by a result of Mansfield and Solovay, in $\H$, for every $x\in A$, the set $\{ y\in A: y\leq^* x\}$ has the perfect set property. This then easily gives that $\H\models CH$. 
\end{proof}

 A positive answer to third question is the content of the Mouse Set Conjecture. 

\subsection{The Mouse Set Conjecture}\label{intro to msc}

MSC, which is stated in the context of $ AD^{+}$, conjectures that \textit{Mouse Capturing} (MC) is true. \\

\textbf{Mouse Capturing:} For all reals $x$ and $y$, $x$ is ordinal definable from $y$ if and only if there is a mouse $\M$ over $y$ such that $x\in \M$.\\

\textbf{The Mouse Set Conjecture:} Assume $AD^{+}$ and that there is no mouse with a superstrong cardinal. Then MC holds.\\

The line of thought, which was due to Martin, Steel and Woodin, that lead to MSC was described in \rsec{levy hierarchy and mice} and had to do with capturing projective truth inside mice (for instance, see \rthm{capturing projective truth}). Rudominer, in  \cite{Rud2}, \cite{Rud3} and \cite{Rud1}, proved  earlier versions of MSC that were beyond projective truth but below $L(\mathbb{R})$-truth. Bulding on this work, Woodin showed that

\begin{theorem}[\cite{TWMS}] $L(\mathbb{R})\models MC$ provided $AD$ holds in $L(\mathbb{R})$.
\end{theorem} 
By doing so, he also developed some of the basic machinery that is used today in many of the arguments in this area of set theory. 
The statement of MSC, as stated above, first appeared in \cite{KoelWoodin} and in \cite{DMATM} around the same time. 

MC essentially says that the most complicated notion of definability, which is the ordinal definability, can be captured by canonical models of fragments of ZFC, which is what mice are. Notice that MC cannot be proved from large cardinals as it implies that $\H\models CH$ while, using forcing, it is possible to arrange a situation where CH fails in $\H$ and the universe has as large cardinals as we wish. This observation shows that MC should be stated in a definable context. From this point of view, MC is the strongest version of such capturing results. The hypothesis that there is no mouse with a superstrong cardinal is needed just because the notion of a mouse isn't well-defined much beyond superstrong cardinals.

Recall that the statement that ``a real $x$ is in a mouse" is $\Sigma^2_3$ (see \rcor{complexity of reals}). Under $AD$, because $\omega_1$ is measurable, $\omega_1$-iterability implies $\omega_1+1$-iterability. Using this observation, it is not hard to see that under $AD$, the statement ``a real $x$ is in a mouse" is $\Sigma^2_1$ as we only need to claim the existence of a set of reals coding an $\omega_1$-iteration strategy. It then follows that $AD^{+}+MC$ implies that the set of ordinal definable reals is just $\Sigma^2_1$, a fact that at first glance might seem utterly implausible.  However, \rthm{sigma21} provides an independent confirmation of this fact without MC. 

\subsection{Directed systems of mice}\label{directed systems of mice}

Steel, in \cite{Steel1995}, developed the basic machinery for analyzing $V_\Theta^{\H}$. Working in $L(\mathbb{R})$, he showed that $V_\Theta^\H$ is a premouse. Building on Steel's work, Woodin then gave the full characterization of $\H^{L(\mathbb{R})}$ by showing that $V_\Theta^\H$ is an iterate of $\M_\omega$ and the full $\H$ is a kind of hybrid structure, a model constructed from an extender sequence and an iteration strategy. The ideas of Steel and Woodin form the basis of the study of $\H$ of models of $AD^+$ and we take a moment to describe this fundamental work. 

First we will need a stronger form of iterability. Suppose $\M$ is a mouse and $\k, \l$ are two ordinals. Then $\mathcal{G}_{\kappa, \l}(\M)$ is a two player game on $\M$ which has $<\k$-rounds. Each round ends with a last model. Let $\M_\a$ be the last model of round $\a$. We let $\M_0=\M$. The $\a+1$'st round is a run of $\mathcal{G}_{\l}(\M_\a)$. Player $I$ exits the rounds and starts a new one. At the end of each round, we have an iteration embedding $\pi_{\a, \a+1}:\M_\a\rightarrow \M_{\a+1}$. The model at the beginning of the $\a$th round for limit $\a$ is the direct limit of $\M_\b$'s for $\b<\a$ computed under the composition of $\pi_{\b, \b+1}$'s. $II$ wins if all the models produced during the run are wellfounded. We say $\M$ is $(\k, \l)$-iterable if $II$ has a wining strategy in $\mathcal{G}_{\k, \l}(\M)$. 

We say $\Sigma$ is a $(\k, \l)$-iteration strategy for $\M$ if $\Sigma$ is a wining strategy for $II$ in $\mathcal{G}_{\k, \l}(\M)$. The runs of $\mathcal{G}_{\k, \l}(\M)$ are called stacks of iteration trees and denoted by $\vec{\T}$. We say $\VT$ is below $\eta<o(\M)$ if all extenders used in $\VT$ have length less than the image of $\eta$, i.e., if $\N$ is a model in $\VT$, $i: \M\rightarrow \N$ is the iteration embedding and $E\in \N$ is the next extender $I$ plays then $lh(E)<i(\eta)$. If $\Sigma$ is a $(\k, \l)$-strategy for $\M$ then we say $\N$ is a $\Sigma$-iterate of $\M$ below $\eta$ via $\VT$ if $\VT$ is a run of $\mathcal{G}_{\k, \l}(\M)$ in which $II$ played according to $\Sigma$, $\VT$ is below $\eta$ and $\N$ is the last model of $\VT$. We may also say $\N$ is a $\Sigma$-iterate of $\M$ below $\eta$ to mean that there is $\VT$ as above.


Assume that $\M_\omega$ exists and is $(\omega_1, \k)$-iterable for some $\k>\card{\mathbb{R}}$. Let $\M=\M_\omega$ and let $\Sigma$ be an $(\omega_1, \k)$-iteration strategy for $\M$. If $\P$ is a $\Sigma$-iterate of $\M$ then we let $\d_i^\P$ be the $i$th Woodin cardinal of $\P$ and $\d_\omega^\P$ be the sup of the Woodins of $\P$. By a Skolem hull argument it can be shown that $\d_\omega^\M$ is countable. We let
\begin{center}
$\mathcal{F}=\{\P:\P$ is a $\Sigma$-iterate of $\M$ below $\d_0^\M$ such that $\d_\omega^\P$ is countable$\}$.
\end{center}
Notice that each $\P\in \mathcal{F}$ inherits an $(\omega_1, \k)$-iteration strategy from $\Sigma$. This is because each run of $\mathcal{G}_{\omega_1, \kappa}(\P)$ can be stimulated as a run of $\mathcal{G}_{\omega_1, \k}(\M)$. We let $\Sigma_\P$ be this iteration strategy. Define $\leq^*$ on $\mathcal{F}$ by letting for $\P, \Q\in \mathcal{F}$,
\begin{center}
$\P\leq^*\Q \iff \Q$ is a $\Sigma_\P$-iterate of $\Q$. 
\end{center}
If $\P\leq^*\Q$ then let $\pi_{\P, \Q}:\P\rightarrow \Q$ be the iteration embedding given by $\Sigma_\P$. It can be shown $\pi_{\P, \Q}$ doesn't depend on the particular iterations producing $\P$ and $\Q$ (see Dodd-Jensen lemma in \cite{OIMT}).The following is a crucial fact.

\begin{lemma}  \label{direct limit in l(r)} Suppose $\P, \Q, \R\in \mathcal{F}$.
\begin{enumerate}
\item $\leq^*$ is directed.
\item $\leq^*$ is well-founded.
\item If $\P\leq^*\Q\leq^*\R$ then $\pi_{\P, \R}=\pi_{\Q, \R}\circ \pi_{\P, \Q}$.
\end{enumerate}
\end{lemma}

Part 1 and 2 of \rlem{direct limit in l(r)} are consequences of the comparison theorem. Part 3 follows from the Dodd-Jensen lemma.  It follows from \rlem{direct limit in l(r)} that we can construct the direct limit of $(\mathcal{F}, \leq^*)$ unders the maps $\pi_{\P, \Q}$. We let
\begin{center}
$\M_\infty=dirlim(\mathcal{F}, \leq^*)$
\end{center}
where the direct limit is computed under the maps $\pi_{\P, \Q}$. $\M_\infty$ too, just like the points in $\mathcal{F}$, inherits an iteration strategy $\Sigma_\infty$ from $\Sigma$. 

Given a transitive set $N$, let $o(N)= N\cap Ord$. If $N=L_\a[\vec{A}]$ where $\vec{A}$ is a sequence of sets of ordinals and $\xi\leq \a$ then we let $N||\xi$ be $N$ \textit{cut off} at $\xi$, i.e., $N||\xi=(L_\xi[\vec{A}], A, \in)$ where $A=A_\xi$. We let $N|\xi$ be $N||\xi$ without $A$.

\begin{theorem}[Steel, \cite{Steel1995}, \cite{OIMT}] Assume $\M_\omega$ exists and let $\mathcal{H}=\H^{L(\mathbb{R})}$. Let $\d$ be the least ordinal such that $L_\d(\mathbb{R})\prec_1 L(\mathbb{R})$.  Then
\begin{center}
$V_\d^\mathcal{H}=\M_\infty|\d$.
\end{center}
\end{theorem}

Let $\k$ be the sup of the Woodin cardinals of $\M_\infty$. Let $\Lambda_\infty$ be the strategy of $\M_\infty|\Theta$ given by
\begin{center}
$\Lambda_\infty(\T)=\Sigma_\infty(\T)$.
\end{center}
Let $\Lambda=\Lambda_\infty\rest \M_\infty|\k$.

\begin{theorem}[Woodin, \cite{SSW}]\label{hod of l(r)} Assume $\M_\omega$ exists and let $\mathcal{H}=\H^{L(\mathbb{R})}$. Then the following holds.
\begin{enumerate}
\item $\Theta^{L(\mathbb{R})}$ is the least Woodin cardinal of $\M_\infty$.
\item $V_\Theta^\mathcal{H}=\M_\infty|\Theta$.
\item $\mathcal{H}=L[M_\infty, \Lambda]$.
\end{enumerate}
\end{theorem}

It can be shown, using Woodins genericity iterations, that whenever $\M$ is a mouse and $\d$ is its least Woodin cardinal then $\M\models ``\M|\d^+$ is not $\d^++1$-iterable". When coupled with this fact, clause 3 of \rthm{hod of l(r)} then implies that in general one cannot show that $\H$ is a premouse. This observation suggests that in order to analyze $\H$ of bigger models of $AD^{+}$ it is necessary to investigate hybrid structures that are constructed from extender sequences and iteration strategies. It turns out that these structures, when organized appropriately, have all of the  properties that the usual mice have. In particular, they can be compared and they can be used to form directed systems. These directed systems are what converge to $\H$'s of bigger models of $AD^{+}$. In the next few subsections, we will explain how the construction of such directed systems work. We start with \textit{hybrid mice}.


\subsection{Hybrid mice}\label{hybrid mice}

Hybrid mice, introduced by Woodin, have been used to analyze descriptive set theoretic objects via  inner model theoretic methods. Besides having an extender sequence, hybrid mice have another sequence which usually describes an iteration strategy or a mouse operator. In this paper, we will mainly deal with strategy hybrids and within those, we will be mainly concerned with \textit{hod mice}, a special brand of hybrid mice. While reading this subsection it might be helpful to review some of the descriptive set theoretic notions introduced before.

Hybrid premice are structures of the form $L_\a[\vec{E}, \vec{\Sigma}]$ where $\vec{E}$ is an extender sequence and $\vec{\Sigma}$ is a sequence of iteration strategies. There are two kinds of hybrid mice that are most useful. The first type of hybrid premice consist of structures of the form $L_\a[\vec{E}, \vec{\Sigma}]$  where $\vec{\Sigma}$ describes just one iteration strategy $\Sigma$ which is an iteration strategy of some fixed structure $M$. More precisely, given a countable transitive model $M$, we say $M$ is $\k$ iterable if $II$ has a winning strategy in $\mathcal{G}_{\k}(M)$. Here $\mathcal{G}_\k(M)$  is defined the same way as the corresponding game for mice except that $I$ is allowed to choose any extender in $M$\footnote{Strictly speaking the extenders chosen during this iteration should have certain \textit{closure} properties, but the details are irrelevant for us.}. Suppose now $M$ is countable and $\omega_1$-iterable as witnessed by an iteration strategy $\Sigma$. $\N$ is called a $\Sigma$-premouse if $\N$ is a structure of the form $L_\a[\vec{E}, \Sigma]$. In order for premice to have fine structure, one needs to \textit{put} $\Sigma$ on the sequence of $\N$ in a very careful manner. This particular way of feeding the strategy  is explained in great details in \cite{ATHM} and \cite{CMI}. 

A countable $\M$ is called a $\Sigma$-mouse if $II$ has a winning strategy $\Lambda$ in $\mathcal{G}_{\omega_1+1}(\M)$ such that whenever $\N$ is a model appearing in a run of $\mathcal{G}_{\omega_1+1}(\M)$ in which  $II$ plays according to $\Lambda$, $\N$ is a $\Sigma$-premouse. Again, under $AD$, if $M$ is countable, $\Sigma$ is an $\omega_1$-strategy for $M$ and $\M$ is a countable $\Sigma$-premouse then $\M$ is $\omega_1+1$-iterable iff it is $\omega_1$-iterable.

If $(M, \Sigma)$ is as in the previous paragraph and $X$ is a self well-ordered set such that $M\in X$ then we say $\M$ is a $\Sigma$-premouse over $X$ if $\M$ has the form $L_\a[\vec{E}, \Sigma][X]$. $\Sigma$-mice over $X$ are defined similarly.

The second type of hybrid premice consist  of structures $L_\a[\vec{E}, \vec{\Sigma}]$ where $\vec{\Sigma}$ describes iteration strategies for the initial segments of $L_\a[\vec{E}, \vec{\Sigma}]$. These structures, called \textit{layered hybrids}, eventually evolve to become \textit{hod mice}. 

\subsection{Hod mice}\label{hod mice}

Hod mice, which are specifically designed to compute $\H$'s of models of $AD^{+}$, feature prominently in the proof of the Main Theorem. One of the motivations behind their definition is \rthm{woodins in hod}. A hod mouse, besides having an extender sequence, is also closed under the iteration strategies of its own initial segments. These initial segments are called \textit{layers} and they keep track of the places new strategies are activated. More precisely, given a hod premouse $\P$, $\eta$ is called a \textit{layer} of $\P$ if the strategy of $\P|\eta$ is activated at a stage $\a$ for some $\a<(\eta^{+})^\P$. \textit{There is one important exception. All hod mice have a last layer for which no strategy is activated.} All hod mice satisfy $ZFC-Replacement$ and they have exactly $\omega$-more cardinals above the last layer. See \rfig{hod premouse} for a generic picture of a hod premouse. 

\begin{figure}
\begin{center}
\tiny{
$\xymatrix@C=15pt{
& & & & & & & & & \d_\l^{+\omega}& & & & & & & &\\
& & & & & & & & & \d_{\l}, \Sigma_{<\l} \ar@{.}[u] \ar@/_1pc/[u]_{\Sigma_{<\l}}& & & & & & &\\
& & & & & & & & & & & & & & & &\\
& & & & & & & & &\d_{\a+1}, \Sigma_{\a+1} \ar@{.}[uu]& & & & & & &\\
& & & & & & & & &\d_{\a}, \Sigma_{\a} \ar@{.}[u] \ar@/_1pc/[u]_{\Sigma_{\a}}& & & & &\\
& & & & & & & & & \d_1, \Sigma_1\ar@{.}[uu] & & & & & & & &\\
& & & & & & & & & \d_0, \Sigma_0 \ar@{.}[u] & & & & & & & & &\\
& & & & & & & & & \ar@{-}[uuuuuuulll] \ar@{-}[uuuuuuurrr] \ar@{.}[u]& & & & & & & &\\}$}
\end{center}
$\Sigma_{<\l}$ is the strategy of $\P|\d_\l$.
\caption{Hod premouse.}
\label{hod premouse}
\end{figure}

Unlike ordinary mice, the hierarchy of hod mice grows according to the Solovay hierarchy. Currently the theory of hod mice is developed and well-understood only for theories that are weaker than $AD_{\mathbb{R}}+``\Theta$ is regular" and somewhat beyond. As is shown in \cite{ATHM}, this minimality condition translates into a first order property of the hod mouse itself: if $\P$ is a hod mouse such that in it some $\d$ is an inaccessible limit of Woodin cardinals then the derived model of $\P$ at $\d$ satisfies $AD_{\mathbb{R}}+``\Theta$ is regular". It follows that the existence of a hod mouse with an inaccessible limit of Woodin cardinals is beyond $AD_{\mathbb{R}}+``\Theta$ is regular". In a sense such a hod mouse corresponds to the ``sharp" of the minimal model of $AD_{\mathbb{R}}+``\Theta$ is regular". \textit{In this paper, unless we specify otherwise, by ``hod mouse" we mean a hod mouse which doesn't have an inaccessible limit of Woodin cardinals.} 

Iterability for hod mice is a stronger notion than for ordinary mice. Essentially, we need to require that the external iteration strategy of a hod mouse is consistent with the internal one. First, recall the iteration game $\mathcal{G}_{\k, \l}(\M)$ defined in \rsec{directed systems of mice}. Notice that this game makes sense even when $\M$ is a hybrid structure. Also, recall that if $\M$ is countable and is $(\omega_1, \k)$-iterable via iteration strategy $\Sigma$ and $\N$ is a countable $\Sigma$-iterate of $\M$ then $\N$ inherits an $(\omega_1, \k)$-iteration strategy from $\Sigma$. This strategy might depend on the particular run of the game producing $\N$.  \textit{In what follows, we will only consider iteration strategies for which the strategy $\N$ inherits from $\Sigma$ is independent of the run of the game producing $\N$. All iteration strategies constructed in \cite{ATHM} have this property.} We let $\Sigma_\N$ be the unique strategy of $\N$ that it inherits from $\Sigma$.

\begin{definition}
Suppose $\P$ is a hod premouse. Then $\Sigma$ is an $(\omega_1, \omega_1)$-iteration strategy for $\P$ if $\Sigma$ is a wining strategy for II in $\mathcal{G}_{\omega_1, \omega_1}(\P)$ and whenever $\Q$ is a $\Sigma$-iterate of $\P$, $\Sigma^\Q=\Sigma_{\Q}\rest \Q$. If $\P$ is a hod mouse and $\Sigma$ is its $(\omega_1, \omega_1)$-strategy then $(\P, \Sigma)$ is called a \textit{hod pair}. 
\end{definition}

Hod mice have a certain peculiar pattern. All hod mice have Woodin cardinals. The layers of a hod mouse are its Woodin cardinals and their limits. Suppose now $\P$ is a hod premouse. We let $\la \d^\P_\a: \a\leq \l^\P\ra$ be the enumeration of its layers in increasing order. Thus, $\la \d^\P_\a: \a\leq \l^\P\ra$ is the sequence of Woodin cardinals and their limits. The strategies of hod premice are activated in a very careful manner. At stage $\a$ the strategy that is being activated is the strategy of certain $\P(\a)\trianglelefteq\P$. $\P(\a)$ is itself a hod premouse and \textit{it is not the same} as $\P|\d_\a^\P$.

Given a hod premouse $\P$, there are three possible scenarios. (1) $\l^\P$ is a successor ordinal, (2) $\cf^\P(\l^\P)$ isn't a measurable cardinal in $\P$ and (3) $\cf^\P(\l^\P)$ is measurable cardinal in $\P$. (1) splits into two cases. We can have that $\l^\P$ is a successor ordinal and (1.1) $\l^\P-1$ is either a successor or $\cf^\P(\l^\P-1)$ is not a measurable cardinal in $\P$ or (1.2) $\l^\P-1$ is a limit ordinal such that $\cf^\P(\l^\P-1)$ is a measurable cardinal in $\P$. We will examine all these cases via pictures. The following notion will be used to define the $\P(\a)$'s more precisely.

If $\M$ is a ($\Sigma$) mouse over $X$ then we say $\M$ is \textit{a sound $(\Sigma)$ mouse over $X$ projecting to $X$} if there is a finite sequence of ordinals $p\in \M$ such that there is a surjection $f: X\rightarrow M$ which is definable over $\M$ from parameters in the set $p\cup \{ X\}$. This isn't the actual definition of a ``sound mouse over $X$ projecting to $X$". The actual definition involves various fine structural notions and it is beyond the scope of this paper. 

\begin{lemma}
If $\M$ and $\N$ are two sound ($\Sigma$) mice over $X$ projecting to $X$ then $\M\trianglelefteq \N$ or $\N\trianglelefteq\M$.
\end{lemma}

\begin{definition}\label{lp} Suppose $\Sigma$ is an iteration strategy for some structure $M$, $a$ is a transitive self-wellordered set such that $M\in a$,  $\kappa=\card{a}^+$, and $\Sigma$ is a $\k$-iteration strategy. Then
\begin{enumerate}
\item $Lp^\Sigma_0(a)=a\cup \{a\}$
\item for $\a<\k$, $Lp^\Sigma_{\a+1}(a)=\cup \{ \N: \N$ is a sound $\k+1$-iterable $\Sigma$-mouse over $Lp^\Sigma_{\a}(a)$ projecting to $Lp^\Sigma_\a(a)\}$.
\item for $\l\leq \k$, $Lp^\Sigma_\l(a)=\cup_{\a<\l}Lp^\Sigma_\a(a)$.
\end{enumerate}
\end{definition}
We let $Lp^\Sigma(a)=Lp^\Sigma_1(a)$.

We now consider all the possible types of hod mice mentioned above. \rfig{hod premouse} is a picture of a hod mouse in general. There are new ordinal parameters denoted by $\mu_\a$ that we use in our pictorial definition of a hod premouse. These ordinals indicate the height of $\P(\a)$'s.

\textbf{Case $\l^\P=0$.} In this case, there is a single Woodin cardinal in $\P$ and there are no strategies in $\P$, i.e., $\P$ is an ordinary premouse such that $\d_0^\P$ is the unique Woodin cardinal of $\P$, $\P(0)=\P$ and $\mu_0=o(\P)$. If $\d$ is the Woodin cardinal of $\M_1^\#$, the minimal mouse with a unique Woodin cardinal and a last extender, and $\mu=(\d^{+\omega})^{\M_1^\#}$ then $\P=\M_1^\#| \mu$ is a hod mouse such that $\l^\P=0$. See \rfig{hod premouse with lambda=0} for a picture.\\
\begin{figure}
\begin{center}
\small{
$\xymatrix@R=20pt@C=10pt{
& & & & & & & & & & & & &\P, \d_0^{+\omega}=\mu_0& & & & & & & &\\
& & & & & & & & & & & & & & & & & & & & & & & & & & & \\
& & & & & & & & & & & & & \d_0 \ar@{.}[uu] & & & & & & & & &\\
& & & & & & & & & & & & & & & & & & & & & & & & & & & \\
& & & & & & & & & & & & & & & & & & & & & & & & & & & \\
& & & & & & & & & & & & & \ar@{-}[uuuuul] \ar@{-}[uuuuur] \ar@{.}[uuu]& & & & & & & &\\}$}
\end{center}
\caption{Hod premouse with $\l^\P=0$, $\P(0)=\P$ and $\mu_0=o(\P)$.}
\label{hod premouse with lambda=0}
\end{figure}

\textbf{Case $\l^\P=1$.} In this case, there are two Woodin cardinals and one strategy in $\P$. $\d_0^\P$ and $\d_1^\P$ are the Woodin cardinals of $\P$, $\P(0)=(Lp_\omega(\P|\d_0^\P))^\P$, $\mu_0=o(\P(0))$, $\Sigma_0$ is the strategy of $\P(0)$, $\mu_1=o(\P)$ and $\P(1)=\P$. Notice that $\P(0)$ is also a hod premouse which is actually an ordinary premouse. $\P_1$ is a $\Sigma_0$-premouse over $\P_0$ and $\P_1=(Lp_\omega^{\Sigma_0}(\P|\d_1))$. See \rfig{hod premouse with lambda=1} for a picture.
\begin{figure}
\begin{center}
\tiny{
$\xymatrix@R=10pt@C=5pt{
& & & & & & & & & & & & \P,\d_1^{+\omega}=\mu_1& & & & & & & &\\
& & & & & & & & & & & & \d_1 \ar@{.}[u] & & & & & & & &\\
& & & & & & & & & & & & & & & & & & & & & & & & & & & & & & & & & & & & \\
& & & & & & & & & & & & & & & & & & & & & & & & & & & & & & & & & & & & \\
& & & & & & & & & & & & \P(0),\mu_0, \Sigma_0 \ar@{.}[uuu]\ar@/_1pc/[uuuu]_{\Sigma_0-mouse} & & & & & & \mu_0=o(Lp_\omega(\P| \d_0)), \P(0)=\P|\mu_0. & & & & & & & & & & & & & \\
& & & & & & & & & & & & \d_0 \ar@{.}[u] & & & & & & & & &\\
& & & & & & & & & & & & & & & & & & & & & & & & & & & & & & & & & & & & \\
& & & & & & & & & & & & & & & & & & & & & & & & & & & & & & & & & & & & \\
& & & & & & & & & & & & \ar@{-}[uuuuuuuulllll]\ar@{-}[uuuuuuuurrrrr] \ar@{.}[uuu]& & & & & & & &\\}$}
\end{center}
\caption{Hod premouse with $\l^\P=1$.}
\label{hod premouse with lambda=1}
\end{figure}

\textbf{Case $\l^\P=\omega$.} In this case, there are $\omega$ many Woodin cardinals and $\omega$ many strategies in $\P$. The $\d_n^\P$'s are the Woodin cardinals of $\P$, $\d_\omega^\P$ is the sup of the Woodin cardinals of $\P$, $\P(0)=(Lp_\omega(\P|\d_0^\P))^\P$, $\P(n+1)=(Lp^{\Sigma_{n}}_\omega(\P| \d_{n+1}^\P))^\P$ where $\Sigma_n$ is the strategy of $\P(n)$, $\mu_n=o(\P(n))$ and $\P=\P(\omega)=(Lp_\omega^{\Sigma_{<\omega}}(\P|\d_\omega))^\P$ where $\Sigma_{<\omega}=\oplus_{n<\omega}\Sigma_n$. Notice that the $\P(n)$'s are also hod premice. See \rfig{hod premouse with lambda=omega} for a picture.
\begin{figure}
\begin{center}
\tiny{
$\xymatrix@R=6pt@C=5pt{
   & & & & & & & & & \P,\d_\omega^{+\omega}=\mu_\omega& & & & & & & & & &\\
   & & & & & & & & & & & & & & & & & & & & & & & & & & & & & & & & & & &\\
   & & & & & & & & & \d_{\omega}, \Sigma_{<\omega}=\oplus_{n<\omega}\Sigma_n \ar@{.}[uu]\ar@/_1pc/[uu]_{\Sigma_{<\omega}-mouse}  & & & & & & & & & & & & & & & & & & & & & & & \\
   & & & & & & & & & & & & & & & & & & & & & & & & & & & & & & & & & \\
   & & & & & & & & & & & & & & & & & & & & & & & & & & & & & & & & & \\
   & & & & & & & & & \P(n+1),\mu_{n+1}, \Sigma_{n+1} \ar@{.}[uuu]& & & & & & & & &\mu_{n+1}=o(Lp^{\Sigma_{n}}_\omega(\P|\d_{n+1})), \P(n+1)=\P|\mu_{n+1} & & & & & & & & & & & & & & \\
   & & & & & & & & & \d_{n+1} \ar@{.}[u]& & & & & & & & & & & & & & & & & & & & & & & & \\
   & & & & & & & & & & & & & & & & & & & & & & & & & & & & & & & & & \\
   & & & & & & & & & \P(n),\mu_n, \Sigma_n \ar@{.}[uu]\ar@/_1pc/[uuu]_{\Sigma_n-mouse}& & & & & & & & & & & & & & & & & & & & & & & & & & & \\
   & & & & & & & & & \d_n\ar@{.}[u] & & & & & & & & & & & & & & & & & & & & & & & & \\
   & & & & & & & & & & & & & & & & & & & & & & & & & & & & & & & & & \\
   & & & & & & & & & & & & & & & & & & & & & & & & & & & & & & & & \\
   & & & & & & & & & \P(1),\mu_1, \Sigma_1 \ar@{.}[uuu]& & & & & & & & &\Sigma_1 \ is\ activated \ here \ar@/_1pc/[lllllllll] & & & & & & & & & & & & & & & & \\
   & & & & & & & & & \d_1 \ar@{.}[u] & & & & & & & &\\
   & & & & & & & & & & & & & & & & & & & & & & & & & & & & & & & & & \\
   & & & & & & & & & \P(0),\mu_0, \Sigma_0 \ar@{.}[uu] & & & & & & & & &\Sigma_0 \ is\ activated \ here\ar@/_1pc/[lllllllll] & & & & & & & & & & & & \\
   & & & & & & & & & \d_0 \ar@{.}[u] & & & & & & & & &\\
   & & & & & & & & & & & & & & & & & & & & & & & & & & & & & & & & & \\
   & & & & & & & & & \ar@{-}[uuuuuuuuuuuuuuuuuullllllll]\ar@{-}[uuuuuuuuuuuuuuuuuurrrrrrrr] \ar@{.}[uu]& & & & & & & &\\}$}
\end{center}
\caption{Hod premouse with $\l^\P=\omega$.}
\label{hod premouse with lambda=omega}
\end{figure}

\textbf{Case $\l^\P$ has a measurable cofinality.} We take the simple case when $\P$ is the hod premouse with the property that $\cf^\P(\l^\P)$ is the least measurable cardinal of $\P$. Let $\kappa$ be the least measurable of $\P$. Then $\P$ has $\kappa$ many Woodin cardinals and $\kappa$ many strategies. In this case, for $\a<\k$, the $\d_{\a+1}^\P$'s are the Woodin cardinals of $\P$ and for limit $\a<\k$, the $\d_\a^\P$'s are limits of Woodin cardinals. As before, $\P(0)=(Lp_\omega(\P|\d_0^\P))^\P$, $\P(\a+1)=(Lp^{\Sigma_{\a}}_\omega(\P|\d_{\a+1}^\P))^\P$ where $\Sigma_\a$ is the strategy of $\P(\a)$, for limit $\a$,  $\P(\a)=(Lp^{\Sigma_{<\a}}_\omega(\P|\d_{\a}^\P))^\P$ where $\Sigma_\a=\oplus_{\b<\a}\Sigma_\b$, $\mu_\a=o(\P(\a))$ and $\P=\P(\k)$. In this case, $\l^\P=\k$. Notice that the $\P(\a)$'s are also hod premice. See \rfig{hod premouse with lambda having a measurable cofinality} for a picture.

\begin{figure}
\begin{center}
\tiny{
$\xymatrix@R=7pt@C=4pt{
& & & & & & & & \P,\d_\k^{+\omega}=\mu_\k& & & & & & & &\\
& & & & & & & & & & & & & & & & & & & & & & & & & & & & & & & & & & & & \\
& & & & & & & & \d_{\k}, \Sigma_{<\k}=\oplus_{\b<\k}\Sigma_\b \ar@{.}[uu]\ar@/_1pc/[uu]_{\Sigma_{<\k}-mouse}  & & & & & & & & & & & & & & & & & & & & & & & \\
& & & & & & & & & & & & & & & & & & & & & & & & & & & & & & & & & & & & \\
& & & & & & & & & & & & & & & & & & & & & & & & & & & & & & & & & & & & \\
& & & & & & & & & & & & & & & & & & & & & & & & & & & & & & & & & & & & \\
& & & & & & & & \P(\a+2), \mu_{\a+2}, \Sigma_{\a+2}\ar@{.}[uuuu]& & & & & & & & & \mu_{\a+2}=o(Lp^{\Sigma_{\a+1}}_\omega(\P|\d_{\a+2})), \P(\a+2)=\P|\mu_{\a+2}& & & & & & & & & & & & & & & & & & & & \\
& & & & & & & & \d_{\a+2}\ar@{.}[u]& & & & & & & & & & & & & & & & & & & & & &\\
& & & & & & & & & & & & & & & & & & & & & & & & & & & & & & & & & & & & \\
& & & & & & & & \P(\a+1),\mu_{\a+1}, \Sigma_{\a+1} \ar@{.}[uu]\ar@/_1pc/[uu]_{\Sigma_{\a+1}-mouse}& & & & & & & & &\mu_{\a+1}=o(Lp^{\Sigma_{\a}}_\omega(\P|\d_{\a+1})), \P(\a+1)=\P|\mu_{\a+1} & & & & & & & & & & & & & & \\
& & & & & & & & \d_{\a+1} \ar@{.}[u]& & & & & & & & & & & & & & & & & & & & & & & & \\
& & & & & & & & & & & & & & & & & & & & & & & & & & & & & & & & & & & & \\
& & & & & & & & \P(\a),\mu_\a, \Sigma_\a \ar@{.}[uu]\ar@/_1pc/[uuu]_{\Sigma_\a-mouse}& & & & & & & & & & & & & & & & & & & & & & & & & & & \\
& & & & & & & & & & & & & & & & & & & & & & & & & & & & & & & & & & & & \\
& & & & & & & & & & & & & & & & & & & & & & & & & & & & & & & & & & & & \\
& & & & & & & & \ar@/_.3pc/[uuu]_{\Sigma_{<\a}-mouse}\d_\a, \Sigma_{<\a}=\oplus_{\b<\a}\Sigma_\b\ar@{.}[uuu] & & & & & & & & \ar@/_1pc/[llllllll]\a\ limit & & & & & & & & & & & & & & & & & & & & & & \\
& & & & & & & & & & & & & & & & & & & & & & & & & & & & & & & & & & & & \\
& & & & & & & & & & & & & & & & & & & & & & & & & & & & & & & & & & & & \\
& & & & & & & & \P(0),\mu_0, \Sigma_0 \ar@{.}[uuu] & & & & & & & & & & & & & & & & & & & \\
& & & & & & & & \d_0 \ar@{.}[u] & & & & & & & & &\\
& & & & & & & & & & & & & & & & & & & & & & & & & & & & & & & & & & & & \\
& & & & & & & & & & & & & & & & & & & & & & & & & & & & & & & & & & & & \\
& & & & & & & & \k \ar@{.}[uuu]& & & & & & & & &\ar@/_1.5pc/[lllllllll] the\ least\ measurable\ cardinal\ of\ \P & & & & & & & & & & & & & & & & & & & & & \\
& & & & & & & & \ar@{-}[uuuuuuuuuuuuuuuuuuuuuuullllllll]\ar@{-}[uuuuuuuuuuuuuuuuuuuuuuurrrrrrrr] \ar@{.}[u]& & & & & & & &\\}$}
\end{center}
\caption{Hod premouse with $\P\models ``\l^\P=$the least measurable cardinal $\k$".}
\label{hod premouse with lambda having a measurable cofinality}
\end{figure}

\textbf{Case $\l^\P$ is a successor and $\l^\P-1$ is a limit of measurable cofinality.} We take the simple case when $\P$ is the least hod premouse with the property that $\l^\P$ is a successor and $\l^\P-1$ is a limit of measurable cofinality. Let $\kappa$ be the least measurable of $\P$. Then $\l^\P=\k+1$ and $\P$ has $\kappa+1$ many Woodin cardinals and $\kappa+1$ many strategies. In this case, for $\a\leq\k$, the $\d_{\a+1}^\P$'s are the Woodin cardinals of $\P$ and for limit $\a\leq\k$, the $\d_\a^\P$'s are limits of Woodin cardinals. As before, $\P(0)=(Lp_\omega(\P|\d_0^\P))^\P$, $\P(\a+1)=(Lp^{\Sigma_{\a}}_\omega(\P|\d_{\a+1}^\P))^\P$ where $\Sigma_\a$ is the strategy of $\P(\a)$, for limit $\a$, $\P(\a)=(Lp^{\Sigma_{<\a}}_\omega(\P|\d_{\a}^\P))^\P$ where $\Sigma_{<\a}=\oplus_{\b<\a}\Sigma_\b$, $\mu_\a=o(\P(\a))$ and $\P=\P(\k+1)$. In this case, the thing to keep in mind is that $\P\models \d_\k^+=(\d^+_\k)^{\P(\k)}$. This is important because unlike the case when cofinality of $\l^\P-1$ isn't measurable, $\Sigma_\k$ is not the same as $\Sigma_{<\k}$. Notice that the $\P(\a)$'s are also hod premice. See \rfig{hod premouse with lambda successor whose predecessor has a measurable cofinality} for a picture.
\begin{figure}
\begin{center}
\tiny{
$\xymatrix@R=5pt@C=5pt{
& & & & & & & & & & & & & & & & & & & & & & & & & & & & & & & & & & & & \\
& & & & & & & & & \P, \mu_{\k+1}=\d_{\k+1}^{+\omega}& & & & & & & & & & & & & & & & & & & & & & & & & & & & & \\
& & & & & & & & &\d_{\k+1}\ar@{.}[u]& & & & & & & & & & & & & & & & & & & & & & & & & & & & & \\
& & & & & & & & & & & & & & & & & & & & & & & & & & & & & & & & & & & & \\
& & & & & & & & & & & & & & & & & & & & & & & & & & & & & & & & & & & & \\
& & & & & & & & & \P(\k),\mu_\k, \Sigma_{\k} \ar@{.}[uuu]\ar@/_1pc/[uuuu]_{\Sigma_\k-mouse}& & & & & & & &\\
& & & & & & & & & & & & & & & & & & & & & & & & & & & & & & & & & & & & \\
& & & & & & & & & & & & & & & & & & & & & & & & & & & & & & & & & & & & \\
& & & & & & & & &\d_{\k}, \Sigma_{<\k}=\oplus_{\b<\k}\Sigma_\b \ar@{.}[uuu]\ar@/_.5pc/[uuu]_{\Sigma_{<\k}-mouse}  & & & & & &\d_\k^+=(\d_\k^+)^{\P(\k)} & & & & & & & & & & & & & & & \\
& & & & & & & & & & & & & & & & & & & & & & & & & & & & & & & & & & & & \\
& & & & & & & & & & & & & & & & & & & & & & & & & & & & & & & & & & & & \\
& & & & & & & & &\P(0),\mu_0, \Sigma_0 \ar@{.}[uuu] & & & & & & & & & & & & & & & & & & & \\
& & & & & & & & &\d_0 \ar@{.}[u] & & & & & & & & &\\
& & & & & & & & & \k \ar@{.}[u]& & & & & & \ar@/_1.5pc/[llllll] the\ least\ measurable\ cardinal\ of\ \P & & & & & & & & & & & & & & & & & & & & & \\
& & & & & & & & & \ar@{-}[uuuuuuuuuuuuulllll]\ar@{-}[uuuuuuuuuuuuurrrrr] \ar@{.}[u]& & & & & & & &\\}$}
\end{center}
\caption{Hod premouse with $\P\models ``\l^\P=\k+1$ where $\k$ is the least measurable cardinal".}
\label{hod premouse with lambda successor whose predecessor has a measurable cofinality}
\end{figure}

\subsection{The proof of MSC}\label{proof of msc}

In this subsection, we outline the proof of 2 of the Main Theorem. Below it is stated again in an equivalent form. 

\begin{theorem}[\cite{ATHM}]\label{proof of msc thm} Assume $AD^{+}+V=L(\powerset(\mathbb{R}))$ and suppose there is no proper class inner model containing the reals and satisfying $AD_{\mathbb{R}}+``\Theta$ is regular". Then MC holds.
\end{theorem}

Assume $AD^{+}+V=L(\powerset(\mathbb{R}))$. The proof of \rthm{proof of msc thm} is via proving three conjectures that collectively imply MC. These are \textit{the HOD Conjecture (HOC)}, \textit{the Generation of Closed Pointclasses, (GCP)} and \textit{the Capturing of Hod Pairs, (CHP)}. As was mentioned before the notion of a hod mouse isn't well-developed much beyond $AD_{\mathbb{R}}+``\Theta$ is regular". Because of this the statements of the conjectures are somewhat informal, and part of the problem is in extending the theory of hod mice to capture stronger theories from the Solovay hierarchy in a way that the conjectures still hold for such hod mice. 

Given a hybrid premouse $\Q$, we say $\Q$ is a \textit{shortening} of a hod premouse if either
\begin{enumerate}
\item there is a hod premouse $\P$ such that letting $\k$ be the largest layer of $\P$, $\Q=\P|\k$ or
\item for some limit ordinal $\xi$, there is a sequence of hod premice $\la \P_\a: \a<\xi\ra$ such that $\P_\a\triangleleft_{hod}\P_\b$ and $\Q=\cup_{\a<\xi}\P_\a$.
\end{enumerate}

\begin{conjecture}[The Hod Conjecture] \label{the hod conjecture} Assume $AD^{+}+V=L(\powerset(\mathbb{R}))+MC$. Then $V_\Theta^\H$ is a shortening of a hod mouse. Moreover, suppose $\Gamma\subsetneq \powerset(\mathbb{R})$ is such that $\Gamma=\powerset(\mathbb{R})\cap L(\Gamma, \mathbb{R})$. Then $(H_{\Theta^{+\omega}})^{\H^{L(\Gamma, \mathbb{R})}}$ is an iterate of  some countable hod mouse\footnote{Recall that $H_\k$ is the set of all sets of hereditarily size $<\k$.}. 
\end{conjecture}

The $\H$ Conjecture is used to show that the $V_\Theta^\H$ of the initial segment of the Wadge hierarchy that satisfies MC is a shortening of a hod premouse. If MC fails then letting $\Gamma$ be the largest initial segment of the Wadge hierarchy where MC holds, $\H$ Conjecture gives a way of characterizing a set of reals just beyond $\Gamma$ in terms of the iteration strategy of a hod mouse iterating to $(V_\Theta)^{\H^{L(\Gamma, \mathbb{R})}}$. That such a characterization is possible is the content of  GCP. Given a pointclass $\Gamma$, let $w(\Gamma)=\sup\{ w(A) :A\in \Gamma\}$. We say $\Gamma$ is a \textit{closed pointclass} if $\powerset(\mathbb{R})\cap L(\Gamma, \mathbb{R})\subseteq \Gamma$.  

\begin{conjecture}[The Generation of Closed Pointclasses]\label{the generation of full pointclasses} Assume $AD^{+}+V=L(\powerset(\mathbb{R}))$. Suppose $\Gamma \subsetneq \powerset(\mathbb{R})$ is a closed pointclass such that there is a Suslin cardinal $\k>w(\Gamma)$. Suppose $L(\Gamma, \mathbb{R})\models MC$.
Then for some hod pair $(\P, \Sigma)$,
\begin{center}
$w(\Gamma)\leq w(Code(\Sigma))$.
\end{center}
\end{conjecture}

Suppose now that MC fails. It can be shown that there is $\Gamma\subset \powerset(\mathbb{R})$ such that $L(\Gamma, \mathbb{R})\models MC$, $\Gamma=\powerset(\mathbb{R})\cap L(\Gamma, \mathbb{R})$ and if $A$ is such that $w(A)=w(\Gamma)$ then $L(A, \mathbb{R})\models \neg MC$. Let $x$ be a real which is $OD$ but not in a mouse. It then follows from GCP that there is a hod pair $(\P, \Sigma)$ such that $w(\Gamma)\leq w(Code(\Sigma))$. A consequence of this is that $x$ is in some $\Sigma$-mouse. To derive a contradiction, it is shown that $\Sigma$ can be captured by mice. That such a capturing is always possible is the content of CHP. 

\begin{conjecture}[The Capturing of Hod Pairs]\label{the capturing of hod pairs} Suppose $\d$ is a Woodin cardinal and $V_\d$ is $\d+1$-iterable for trees that are in $L_\omega(V_\d)$. Suppose further that there is no mouse with a superstrong cardinal and that $(\P, \Sigma)$ is a hod pair such that $\P\in V_\d$ and $\Sigma$ is a $\d^+$-iteration strategy. Let $\N^*=(L[\vec{E}])^{V_\d}$, the output of the full background construction of $V_\d$, and let $\N=L[\N^*]$. Thus, $\N^*\trianglelefteq\N$. There is then a $\Sigma$-iterate $\Q$ of $\P$ such that if $\Lambda$ is the strategy of $\Q$ induced by $\Sigma$ then $\Q\in \N|\d$ and $\Lambda\rest (V_\d)^\N\in \N$.
\end{conjecture}

Continuing with the above set up, we look for $M$ such that $\P\in M$, $\Sigma\rest M\in M$ and for some $\delta$, $(M, \d, \P, \Sigma\rest M)$ satisfies the hypothesis of \rcon{the capturing of hod pairs}. That there is always such an $M$ is a theorem due to Woodin. It then follows from \rcon{the capturing of hod pairs} that some tail of $\Sigma$ is in a mouse implying that in fact $x$ is in some mouse as well. The details of this rough sketched are worked out in great detail in \cite{ATHM}.  

What is proved in \cite{ATHM} is that the three conjectures together imply MSC and all three conjectures are true under an additional assumption that there is no proper class inner model containing the reals and satisfying $AD_{\mathbb{R}}+``\Theta$ is regular".

\begin{theorem}[\cite{ATHM}]\label{three conjectures} Assume $AD^{+}+V=L(\powerset(\mathbb{R}))$ and that HOC, GCP and CHP are true. Then MC holds.
\end{theorem}

\begin{theorem}[\cite{ATHM}]\label{conjectures} Assume $AD^{+}+V=L(\powerset(\mathbb{R}))$ and suppose there is no proper class inner model containing the reals and satisfying $AD_{\mathbb{R}}+``\Theta$ is regular". Then HOC, GCP and CHP are all true and hence, MC is true as well.
\end{theorem}



\subsection{The comparison theory of hod mice}

While we will encounter uncountable hod premice, all hod mice of this paper are countable. Notice that comparison may not hold for arbitrary two hod pairs. For instance, if $(\P, \Sigma)$ and $(\Q, \Lambda)$ are two hod pairs such that $\l^\P, \l^\Q\geq 1$ then it is possible that in the comparison of $\P(0)$ and $\Q(0)$, $\P(0)$ iterates into a proper initial segment of an iterate of $\Q(0)$. This means that further comparison of $\P$ and $\Q$ is meaningless. In general, we do not know how to compare arbitrary hod pairs. Our comparison theorem works for hod pairs $(\P, \Sigma)$ such that $\Sigma$ has \textit{branch condensation} and is \textit{fullness preserving}. Both are technical properties that we will define at the end of this subsection. It is possible to continue with the paper without a solid understanding of what these notions are. 

Here is what comparison means for hod pairs. Given two hod premice $\P$ and $\Q$, we write $\P\trianglelefteq_{hod}\Q$ if there is $\a\leq\l^\Q$ such that $\P=\Q(\a)$. Given a hod pair $(\P, \Sigma)$ we let
\begin{center}
$I(\P, \Sigma)=\{ \Q: \Q$ is an $\Sigma$-iterate of $\P\}$.
\end{center}

\textbf{Comparison for hod pairs:} Suppose $(\P, \Sigma)$ and $(\Q, \Lambda)$ are two hod pairs. Then comparison holds for $(\P, \Sigma)$ and $(\Q, \Lambda)$ if there are $\M\in I(\P, \Sigma)$ and $\N\in I(\Q, \Lambda)$ such that one of the following holds:
\begin{enumerate}
\item  $\M\trianglelefteq_{hod}\N$ and $(\Lambda_{\N})_\M=\Sigma_{\M}$.
\item  $\N\trianglelefteq_{hod}\M$ and $(\Sigma_{\M})_\N=\Lambda_{\N}$.\\
\end{enumerate}

The following is the comparison theorem proved in \cite{ATHM}.

\begin{theorem}[Comparison, \cite{ATHM}]\label{comparison} Assume $AD^{+}+V=L(\powerset(\mathbb{R}))$. Suppose $(\P, \Sigma)$ and $(\Q, \Lambda)$ are two hod pairs such that both $\Sigma$ and $\Lambda$ have branch condensation and are fullness preserving. Then comparison holds for $(\P, \Sigma)$ and $(\Q, \Lambda)$.
\end{theorem}

Branch condensation essentially says that if any iteration is realized into an iteration via strategy then it is also an iteration according to the strategy. Below is the definition of branch condensation for iteration trees. What is really needed for the comparison theory is a stronger notion for iterations produced via the runs of $\mathcal{G}_{\omega_1, \omega_1}$ but we won't need this version in our current exposition.

\begin{definition}[Branch condensation]\label{branch condensation} Suppose $M$ is a transitive model of some fragment of $ZFC$ and $\Sigma$ is an iteration strategy for $M$. Then $\Sigma$ has branch condensation (see Figure \ref{branch condensation}) if for any two iteration trees $\T$ and $\U$ on $M$ and any branch $c$ of $\U$ if
\begin{enumerate}
\item $\T$ and $\U$ are according to $\Sigma$, 
\item $lh(\U)$ is limit and $lh(\T)=\gg+1$,
\item  for some $\pi : \M_c^{\U} \rightarrow_{\Sigma_1} \M_\gg^\T$,
    \begin{center}
     $i^\T_{0, \gg}=\pi\circ i_c^{\U}$
    \end{center}
\end{enumerate}

 then $c=\Sigma(\vec{\U})$.
\end{definition}

\begin{figure}
\begin{center}\small{
$\xymatrix@1@R=25pt{
& & & \M_\gg^\T\\
\M \ar[urrr]^{\T, i_{0, \gg}^\T} \ar[drrr]_{\U, c, i_c^\U} & & & & if \ i^{\T}_{0, \gg}=\pi\circ i_c^{\U},\ then\ \Sigma(\U)=c\\
& & & \M_c^{\U} \ar[uu]^{\pi}}$}
\end{center}

\caption{Branch condensation}
\label{branch condensation}
\end{figure}
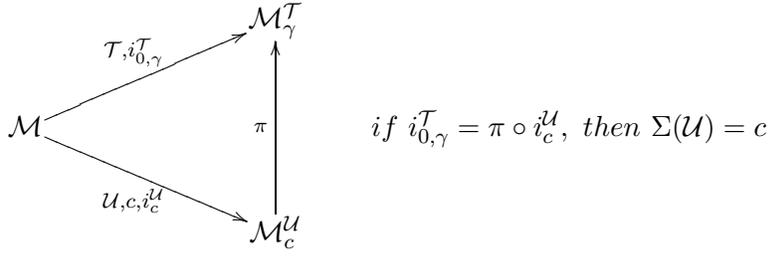

Fullness preservation refers to the degree of correctness of the models. It essentially says that the iterates of the hod mouse contain all the mice present in the universe.  Given a (hod) mouse $\M$ and $\eta$, we say $\eta$ is a strong cutpoint of $\M$ if there is no $\k\leq \eta$ which is $\eta$-strong as witnessed by the extenders on the sequence of $\M$.

\begin{definition}[Fullness Preservation]\label{fp}
Suppose $(\P, \Sigma)$ is a hod pair. $\Sigma$ is fullness preserving if whenever $\Q\in I(\P, \Sigma)$, $\a+1\leq \l^\Q$ and $\eta>\d_\a$ is a strong cutpoint of $\Q(\a+1)$, then
\begin{center}
$\Q|(\eta^+)^{\Q(\a+1)}=Lp^{\Sigma_{\Q(\a)}}(\Q| \eta)$.
\end{center}
and
\begin{center}
$\Q|(\d_\a^+)^\Q=Lp^{\oplus_{\b<\a}\Sigma_{\Q(\b+1)}}(\Q(\a))$.
\end{center}
\end{definition}

One important consequence of branch condensation and fullness preservation is that they imply that $\Sigma$ is \textit{positional} and \textit{commuting}. This is important for direct limit constructions. Given a hod pair $(\P, \Sigma)$, we say $\Sigma$ is positional if whenever $\Q\in I(\P, \Sigma)$ and $\R\in I(\Q, \Sigma_\Q)$ then the iteration embedding from $\Q$-to$\R$ is independent from the particular run of the iteration game producing $\R$. If $\Sigma$ is positional then we let $\pi^\Sigma_{\Q, \R}:\Q\rightarrow \R$ be the unique iteration embedding given by $\Sigma$. It certainly depends on $\Sigma$. We say $\Sigma$ is commuting if whenever $\Q\in I(\P, \Sigma)$, $\R\in I(\Q, \Sigma_\Q)$ and $\S\in I(\R, \Sigma_\R)$ then 
\begin{center}
$\pi_{\Q, \S}^\Sigma=\pi_{\R, \S}^\Sigma\circ \pi^\Sigma_{\Q, \R}$. 
\end{center}

\begin{theorem}[\cite{ATHM}]\label{commuting strategy} Assume $AD^{+}+V=L(\powerset(\mathbb{R}))$ and suppose $(\P, \Sigma)$ is a hod pair such that $\Sigma$ has branch condensation and is fullness preserving. Then $\Sigma$ is both positional and commuting. 
\end{theorem}

Equipped with our comparison theorem we can now explain how the analysis of $\H$ is done. 

\subsection{$\H$ is a hod premouse}\label{hod is a hod premouse}

Hod mice were introduced in order to generalize the computation of $\H$ of $L(\mathbb{R})$ to larger models of $AD^+$. Specifically, they are used to prove theorems like the following.

\begin{theorem}[The $\H$ Theorem]\label{hod theorem}
Assume $AD^{+}+V=L(\powerset(\mathbb{R}))$. Suppose that for every $\Gamma\varsubsetneq \powerset(\mathbb{R})$, $L(\Gamma, \mathbb{R})\models \neg ``AD_{\mathbb{R}}+\Theta$ is regular". Then $V_\Theta^\H$ is a shortening of a hod premouse.
\end{theorem}

In this subsection, we will outline the proof of \rthm{hod theorem}. The proof is via induction and the following is the inductive step.

\begin{lemma}[The Inductive Step]\label{inductive step}  Suppose $\a$ is such that $\theta_\a<\Theta$. There is then a hod pair $(\P, \Sigma)$ such that $\Sigma$ has branch condensation, is fullness preserving and $V_{\theta_\a}^\H$ is a shortening of the direct limit of all $\Sigma$-iterates of $\P$ or, using the notation developed below,
\begin{center}
$\M_\infty(\P, \Sigma)|\theta_\a=V_{\theta_\a}^\H$.
\end{center}
\end{lemma}
\begin{proof}
We sketch the proof. The proof is reminiscent of the proof of \rthm{hod of l(r)}. Fix an ordinal $\a$ as in the hypothesis. First we show that there is a hod pair $(\P, \Sigma)$ such that $\Sigma$ has branch condensation and is fullness preserving and for any set of reals $A$,
\begin{center}
$w(A)<\theta_\a \iff A\leq_w Code(\Sigma)$. 
\end{center}
This is an instance of the generation of pointclasses. Let
\begin{center}
$\mathcal{F}(\P, \Sigma)=I(\P, \Sigma)$
\end{center}
and define $\leq^{\Sigma}$ on $\mathcal{F}(\P, \Sigma)$ by letting
\begin{center}
$\Q\leq^{\Sigma} \R$ iff $\R\in I(\Q, \Sigma_\Q)$.
\end{center}
It follows from comparison theorem that $\leq^\Sigma$ is directed. Also, it follows from \rthm{commuting strategy} that $\Sigma$ is commuting. Using this, we can form a direct limit. We let
\begin{center}
$\M_\infty(\P, \Sigma)$
\end{center}
be the direct limit of $(\mathcal{F}(\P, \Sigma), \leq^\Sigma)$ under the iteration maps $\pi^{\Sigma}_{\Q, \R}$. One then shows  that
\begin{center}
$\M_\infty(\P, \Sigma)|\theta_\a=V_{\theta_\a}^\H$\ \ \ \ \ \ \ \ \ \ \ $(*)$.
\end{center}
The forward inclusion of $(*)$ namely that $\M_{\infty}(\P, \Sigma)\subseteq \H$ might seem less plausible as the definition of $\M_\infty(\P, \Sigma)$ seems to require $(\P, \Sigma)$. However, it follows from comparison, \rthm{comparison}, that $\M_\infty(\P, \Sigma)$ is in fact independent of $(\P, \Sigma)$. The full proof of $(*)$ can be found in \cite{ATHM}.
\end{proof}


The next step is to prove a version of \rlem{inductive step} for $\a$ such that $\theta_\a=\Theta$. This will indeed finish the proof of \rthm{hod theorem}. However, it turns out that in this case we cannot get a pair $(\P, \Sigma)$ as in \rlem{inductive step}. The reason is that such a $\Sigma$ can be used to define a surjection from $\mathbb{R}$ onto $\Theta=\theta_\a$. First for $\Q\in \mathcal{F}(\P,\Sigma)$, let
\begin{center}
$\pi_{\Q, \infty}^{\Sigma}:\Q\rightarrow \M_\infty(\P, \Sigma)$
\end{center}
be the direct limit embedding. Next let $A\subseteq \mathbb{R}$ be the set of reals coding the set
 \begin{center}
 $\{ (\Q, \b): \Q\in \mathcal{F}(\P, \Sigma)$  and $\b\in \Q\}$.
 \end{center}
 Then define $f: A\rightarrow Ord$ by
\begin{center}
$f(\Q, \b)=\pi^{\Sigma}_{\Q, \infty}(\b)$,
\end{center}
Then clearly $\theta_\a\subseteq rng(f)$.
 
 Nevertheless, one can still define a certain directed system whose direct limit is $V_{\Theta}^\H$. The proof splits into two different cases. The first case is when $\Theta=\theta_\a$ for some limit $\a$. First, it follows from comparison and $(*)$ that whenever $(\P, \Sigma)$ is a hod pair such that $\Sigma$ has branch condensation and is fullness preserving then for some $\b$ we have that
\begin{center}
$\M_\infty(\P, \Sigma)|\theta_\b=V_{\theta_\b}^\H$\ \ \ \ \ \ $(**)$
\end{center}
and $w(Code(\Sigma))=\theta_\b$. We then let $\b(\P, \Sigma)=_{def}\b$.

Suppose now $\a$ is a limit and $\Theta=\theta_\a$. In this case, using $(**)$ we get that
\begin{center}
$V_\Theta^\H=\cup\{ \M_\infty(\P, \Sigma)|\theta_{\b(\P, \Sigma)} : (\P, \Sigma)$ is a hod pair such that $\Sigma$ has branch condensation and is fullness preserving$\}$.
\end{center}

The second case, namely that $\Theta=\theta_{\a+1}$, is much harder. We do not have the space to outline it in any great detail. The main difficulty is, as mentioned above,  that we cannot have a pair $(\P, \Sigma)$ such that $\M_\infty(\P, \Sigma)|\Theta=V_\Theta^\H$. The idea, which is originally due to Woodin, is to pretend that there is such a pair $(\P, \Sigma)$ and approximate pieces of $\Sigma$ in $M$. We refer the interested reader to \cite{ATHM} and \cite{Trang} for more details on this case. Finally,  Trang, in \cite{Trang}, building on an earlier work of Woodin and the author, gave a full description of $\H$. However, describing this work is beyond the scope of this paper. 

\subsection{Partial results on inner model problem}\label{partial results on imp}

We mentioned earlier that the theory of hod mice  can be used to get partial results on the inner model problem. Here we would like to outline one approach based on hod mice. As was noted before, hod mice below $AD_{\mathbb{R}}+``\Theta$ is regular" have very limited large cardinal structure. They cannot, for instance, have an inaccessible limit of Woodin cardinals. However, they are closed under complicated iteration strategies and if we could translate them into extenders in a complexity-preserving way then we would get complicated mice. But mice are complicated only because of the large cardinals they have. Hence, we just have to examine what large cardinals exist in this translated mouse.

To realize the vague outline just given, Steel devised a translation procedure such that given a hod mouse $\P$ it translates the strategies on the sequence of $\P$ into extenders in a complexity preserving fashion. Let $\P^S$ be the translation of $\P$.  Steel showed that $\P^S$, which is an inner model of $\P$, can recover $\P$, a fact which implies that $\P$ and $\P^S$ have the same complexity.

In \cite{DMATM}, Steel used his methods to analyze the mouse $\P^S$ where $\P$ is a hod mouse with 2 layers. In unpublished work, Steel also analyzed $\P^S$ for a hod mouse $\P$ which has $\omega$ layers. The minimal such hod mouse corresponds to the minimal model of $AD_{\mathbb{R}}$. Letting $\P$ be the minimal hod mouse with $\omega$ layers, Steel, using an instance of \rthm{hod theorem} applied to the minimal model of $AD_{\mathbb{R}}$, showed that $\P^S\models ZFC+AD_{\mathbb{R}}-hypothesis$. This constitutes one half of \rthm{adr equiconsistency}.

Recently Yizheng Zhu, in \cite{Zhu}, analyzed the model $\P^S$ where $\P$ is the minimal hod mouse that corresponds to the minimal model of $AD_{\mathbb{R}}+``\Theta$ is regular". He showed that $\P^S\models ``ZFC+\Theta$-regular hypothesis". This constitutes one half of \rthm{theta-reg hypo}. 

The situation for $LST$ is rather peculiar. We say $\P$ is an $LST$ hod premouse if in $\P$ there are cardinals $\k$ and $\d$ such that $\P=L[\P|\d]$ and in $\P$,
\begin{enumerate}
\item $\d$ is a Woodin cardinal,
\item $\k$ is $<\d$-strong,
\item $\k$ is a limit of Woodin cardinals,
\end{enumerate} 
It can be shown that the existence of an $LST$-hod mouse is weaker than $LST$. 

\begin{theorem}\label{lst theorem} Assume $LST$. There is then a class size $LST$ hod premouse. 
\end{theorem}

The following is an upper bound for the existence of an $LST$ hod premouse. 

\begin{theorem}\label{lst from wlw} Suppose there is a Woodin limit of Woodins. Then in some generic extension, there is an $LST$ hod premouse. 
\end{theorem}

We make the following conjecture. 

\begin{conjecture}\label{lst conjecture} The following are equiconsistent.
\begin{enumerate}
\item ZFC+there is an $LST$-hod premouse.
\item ZFC+there is a Woodin limit of Woodin cardinals.
\end{enumerate}
\end{conjecture}

In fact, we conjecture that if $\P$ is an $LST$ hod premouse then $\P^S\models$ ``ZFC+there is a Woodin limit of Woodin cardinals". 

\subsection{Descriptive inner model theory}\label{dimt}

As was mentioned before the main technical goal of descriptive inner model theory is to construct mice that capture the universally Baire sets. Exactly how these sets are captured was left unexplained. Below we make it more precise. 

Suppose $\M$ is a mouse and $\Sigma$ is an $\omega_1$-iteration strategy for $\M$. Suppose $\d\in \M$ is a Woodin cardinal of $\M$. We let $\mathbb{B}_\d^\M$ be the \textit{extender algebra} of $\M$ at $\d$. Extender algebra is a poset introduced by Woodin. Because we do not need its exact definition here, we will not define it here. Interested readers can consult \cite{EA} and \cite{OIMT}. The following surprising theorem is one of the most useful theorems in descriptive inner model theory. 

\begin{theorem}[Woodin, \cite{EA}, \cite{OIMT}]\label{genericity iterations} Suppose $x\in \mathbb{R}$. Then there is an iteration tree $\T$ on $\M$ according to $\Sigma$ such that $\T$ has a last model $\N$ and for some $\N$-generic $g\subseteq i^\T(\mathbb{B}_\d^\M)$, $x\in \N[g]$.  
\end{theorem}

Suppose $\M$ is a mouse or a hybrid mouse and $\Sigma$ is an iteration strategy. We say $\Sigma$ is \textit{commuting} if whenever $\N$ is a $\Sigma$-iterate of $\M$ and $\P$ is a $\Sigma_\N$-iterate of $\N$ then the iteration embedding $i:\N\rightarrow \P$ doesn't depend on the particular iterations producing $\N$ and $\P$. It can be shown that many iteration strategies are commuting. In particular, fullness preserving strategies of hod mice that have branch condensation are commuting (see \rthm{commuting strategy}).

Continuing with the above set up, suppose $A\subseteq \mathbb{R}$. We then say $(\M, \d, \Sigma)$ Suslin, co-Suslin captures $A$ if $\M$ is a mouse, $\Sigma$ is an $\omega_1$-iteration strategy for $\M$, $\Sigma$ is commuting and there are trees $T, S\in \M$ such that $\M\models ``(T, S)$ is $\d$-complementing" and
\begin{center}
$x\in A \iff$ whenever $\T, \N$ are as in \rthm{genericity iterations}, $x\in (p[i^\T(T)])^{\N[x]}$. 
\end{center}

By a theorem of Neeman from \cite{OPD}, if $A$ is Suslin, co-Suslin captured by $(\M, \d, \Sigma)$ then $A$ is determined. Also, in models of determinacy, if a set $A$ is Suslin, co-Suslin captured by some $(\M, \d, \Sigma)$ then $A$ is Suslin, co-Suslin. Clearly if we take ``capturing" to be defined in the above sense, then we cannot in general capture all universally Baire sets. For instance, in $L$, we cannot capture all $\Pi^1_1$-sets in the above sense. In general, we can only hope to capture a set $A$ in the above sense if there are more complicated universally Baire sets than $A$. Here, our goal is just to illustrate how such a capturing can be formalized in ZFC and hence, we will give us more room than we really need.  

Let $\mathbb{C}_\k$ be the set of $A\subseteq \mathbb{R}$ such that $A$ is $\k$-universally Baire set such that for some $\k$-universally Baire set $B$, $A$ is Suslin, co-Suslin in $L(B, \mathbb{R})$, $L(B, \mathbb{R})\models AD^+$ and every Suslin, co-Suslin set of $L(B, \mathbb{R})$ is $\k$-universally Baire. Let $\mathbb{M}_\k$ be the set of $A\subseteq \mathbb{R}$ such that for some $\k$-universally Baire set $B$ such that $L(B, \mathbb{R})\models AD^+$, $B$ codes a triple $(\M, \d, \Sigma)$ which Suslin, co-Suslin captures $A$.   

\begin{conjecture}[ZFC]\label{equality} For every $\k\geq \omega_1$, $\mathbb{C}_\k=\mathbb{M}_\k$.
\end{conjecture}

By a result of Martin, Steel and Woodin, under proper class of Woodin cardinals, every universally Baire set is in $\mathbb{C}_\k$. Thus, in many situations, such as under the existence of a proper class of Woodin cardinals, \rcon{equality} implies that all universally Baire sets are captured. 

One way to prove \rcon{equality} is by settling \rcon{the generation of full pointclasses} and by devising a general translation procedure such as the one described in \rsec{partial results on imp}. To see this, let $A\in \mathbb{C}_\k$ and let $B\subseteq \mathbb{R}$ witness this fact. Using \rcon{the generation of full pointclasses} in $L(B, \mathbb{R})$, we can find a hod pair $(\P, \Sigma)$ such that $w(A)<w(Code(\Sigma))$ and $Code(\Sigma)$ is Suslin, co-Suslin in $L(B, \mathbb{R})$. Let $\Q=\P^S$ and let $\Lambda$ be the strategy of $\Q$ induced by $\Sigma$. Then it can be shown that for some $\d$, $(\Q, \d, \Lambda)$ Suslin, co-Suslin captures $A$. But because $Code(\Sigma)$ is Suslin, co-Suslin in $L(B, \mathbb{R})$, it follows that $(\Q, \d, \Lambda)$ is coded by a $\k$-universally Baire set $B$. The reverse direction is similar and follows from the fact that if $A$ is Suslin, co-Suslin captured by a triple $(\M, \d, \Sigma)$ such that $L(Code(\Sigma), \mathbb{R})\models AD^+$ then $A$ is Suslin, co-Suslin in $L(Code(\Sigma), \mathbb{R})$. 

It then follows that the main technical goal of descriptive inner model theory can be reduced to devising general translation procedures and proving \rcon{the generation of full pointclasses}.

\subsection{The core model induction}\label{cmi}

What makes the descriptive inner model theoretic approach to the inner model problem universal is the method known as core model induction. It is a tool for constructing models of theories from the Solovay hierarchy and it has been successfully applied in many different situation. The core model induction was first introduced by Woodin and further developed by Schimmerling, Shindler, Steel and others. In recent years, it has been used to attack the PFA Conjecture and related problems. 
The following are instances of results proven using core model induction. Recall that Todorcevic showed that PFA implies $\neg\square_\k$ for all $\k\geq \omega_1$.  

Before stating the theorems we introduce some notation. Fix a cardinal $\l\geq (2^\omega)^+$. We let $HP^-_\l$ be the set of all hod pairs $(\P, \Sigma)$ such that $\P$ is countable and $\Sigma$ is a $(\l, \l)$-iteration strategy with branch condensation. Recall the notation $Lp^\Sigma(a)$ introduced in \rdef{lp}. The author, in \cite{ATHM}, defined $Lp^\Sigma(\mathbb{R})$. The reason that this case is somewhat tricky is that $\mathbb{R}$ is in general not a self-wellordered set. We then let $HP_\l=\{ (\P, \Sigma)\in HP^-_\l: L(Lp^\Sigma(\mathbb{R}))\models AD^+\}$ and
\begin{center}
$\mathbb{K}_\l=\cup \{Lp^\Sigma(\mathbb{R}) : \exists \P((\P, \Sigma)\in HP_\l)\}$. 
\end{center}

\begin{theorem}[Steel, \cite{PFA}]\label{steel pfa} Assume $\neg\square_\k$ holds for some singular strong limit cardinal $\k$. Then for any uncountable $\l<\k$, $\mathbb{R}\in \mathbb{K}_\l$ and hence, $AD$ holds in $L(\mathbb{R})$. In particular, PFA implies that $AD$ holds in $L(\mathbb{R})$.  
\end{theorem}

The next theorem is a generalization of the previous one.

\begin{theorem}[S., partly \cite{StrengthPFA1}]\label{grigor pfa} Assume $\neg\square_\k$ holds for some singular strong limit cardinal $\k$. Then there is a non-tame mouse. Assume, moreover, that $PFA$ holds. Then for any uncountable $\l<\k$, there is $\Gamma\subseteq \mathbb{K}_\l$ such that $L(\Gamma, \mathbb{R})\models AD_{\mathbb{R}}+``\Theta$ is regular". \end{theorem}

 \begin{remark}[Descriptive inner model theoretic approach to the inner model problem]\label{dimt approach to imp} The descriptive inner model theoretic approach to the inner model problem is really the use of the core model induction with the translation procedures of \rsec{partial results on imp}. This way, while working under wide range of axiomatic systems, we can construct mice with large cardinals by first constructing models from the Solovay hierarchy and then, using translation procedures of \rsec{partial results on imp}, construct mice with large cardinals. For instance, using  \rthm{theta-reg hypo} and \rthm{grigor pfa} we get that PFA implies that there is a  class size premouse satisfying  $``ZFC+\Theta$-regular hypothesis".  \end{remark} 
 
 Unfortunately, we do not have the space to explain the details of how the core model induction works. However, interested readers can consult \cite{CMI} which is an excellent introduction to core model induction. Below we give a short sketch.

 The core model induction is a method for inductively showing that $\mathbb{K}_\k$, for some fixed $\kappa$, has various closure properties. This is achieved by considering several cases some of which depend on deep facts from inner model theory such as \textit{the covering theorem} for the core model $K$. At the beginning of the induction, we fix some target theory $S$ from the Solovay hierarchy. The aim of the induction is to build $\Gamma\subseteq \mathbb{K}_\k$ such that $L(\Gamma, \mathbb{R})\models S$. If we reach such a $\Gamma$ then we stop the induction. Otherwise it goes on in which case we get that $L(\mathbb{K}_\k)\models AD^+$. If $L(\mathbb{K}_\k, \mathbb{R})\models S$ then we have reached our goal and hence, we are done. Otherwise, using our hypothesis (in the case of \rthm{grigor pfa}, PFA) we construct a set $A$ and show that $A\in \mathbb{K}_\k\iff A\not \in \mathbb{K}_\k$, a contradiction implying that $L(\mathbb{K}_\k, \mathbb{R})\models S$. To construct such a set $A$, we use \rcon{the hod conjecture} in $L(\mathbb{K}_\k, \mathbb{R})$ to construct some hod pair $(\P, \Sigma)$ such that the direct limit of all $\Sigma$-iterates of $\P$ contains $(V_{\Theta})^{\H^{L(\mathbb{K}_\k, \mathbb{R})}}$. It is then shown that $Code(\Sigma)\in \mathbb{K}_\k$. However, as is explained in \rsec{hod is a hod premouse}, such a $\Sigma$ can never be in $L(\mathbb{K}_\k, \mathbb{R})$. The resulted contradiction implies that in fact $L(\mathbb{K}_\k, \mathbb{R})\models S$. Because the construction of $\Sigma$ heavily depends on \rcon{the hod conjecture}, core model induction can only reach theories from the Solovay hierarchy for which \rcon{the hod conjecture} has been verified. 
 
We believe that core model induction has a great potential and can be used to settle the forward direction of the PFA Conjecture. While this will probably happen far in the future the following conjectures are within reach.
 
 \begin{conjecture} Assume PFA. Then for every $\k\geq \omega_3$ there is $\Gamma\subseteq \mathbb{K}_\k$ such that $L(\Gamma, \mathbb{R})\models LST$.
 \end{conjecture}
 
 \begin{conjecture}\label{strength from the iterability conjecture} Suppose the Iterability Conjecture fails. Then for every $\k\geq (2^{\omega})^+$ there is $\Gamma\subseteq \mathbb{K}_\k$ such that $L(\Gamma, \mathbb{R})\models LST$.
 \end{conjecture}
 
 Notice that positive answer to \rcon{lst theorem} will imply that both hypothesis are at least as strong as ``ZFC+there is a Woodin limit of Woodins". Also, notice that \rcon{strength from the iterability conjecture} is yet another way of solving the inner model problem via the methods of descriptive inner model theory. In this approach, we work in the context of large cardinals. Suppose we are trying to construct a mouse which satisfies some large cardinal axiom $\phi$. Suppose we have stronger large cardinals in $V$, so strong that if the full background construction converges then it also satisfies $\phi$. If such constructions do not converge then it must be the case that the Iterability Conjecture is false. Now, via core model induction construct some $\Gamma\subseteq \mathbb{K}_\k$ such that $L(\Gamma, \mathbb{R})\models S_\phi$. Let $\Q$ be $\mathcal{H}^S$ where $\mathcal{H}=(V_\Theta)^{\H^{L(\Gamma, \mathbb{R})}}$. Then $\Q\models \phi$ and hence, it is as desired. 

\subsection{Divergent models of $AD^+$}\label{applications}

The first reasonable upper bound for $AD_{\mathbb{R}}+``\Theta$ regular" was obtained in \cite{ATHM} using divergent models of $AD$. The existence of divergent models of $AD$ is an interesting phenomenon in descriptive set theory, and the proof of $AD_{\mathbb{R}}+``\Theta$ regular" from divergent models of $AD$ is of independent interest. Below we describe the proof of this fact, which constitutes part 3 of the Main Theorem.

 \begin{theorem}[\cite{ATHM}]\label{div models}
Suppose  $L(A,\mathbb{R})$ and $ L(B,\mathbb{R})$ are divergent models of $ AD^{+}+V=L(\powerset(\mathbb{R}))$. Then there is $M$ such that $Ord, \mathbb{R}\subseteq M$ and $M\models  AD_{\mathbb{R}}+``\Theta$ is regular".
\end{theorem}
\begin{proof}
Suppose not. Let
 \begin{center}
 $\Gamma=\powerset(\mathbb{R})\cap L(A,\mathbb{R}) \cap L(B,\mathbb{R})$.
 \end{center}
 To continue the proof we will need the notion of $\Gamma$-fullness preservation. Because we would like to make the exposition as non-technical as possible, we will not define $\Gamma$-fullness preservation in details. It is just like the fullness preservation except we require that $\Q$ of \rdef{fp} is full with respect to $\Sigma_{\Q(\a)}$-mice that have iteration strategies coded by a set of reals in $\Gamma$.
 
 Then, applying a finer version of the generation of pointclasses, we can get $(\P, \Sigma)\in  L(A, \mathbb{R})$ and $(\Q, \Lambda)\in  L(B, \mathbb{R})$ such that both $\Sigma$ and $\Lambda$ have branch condensation and are $\Gamma$-fullness preserving. Moreover,
\begin{center}
$w(Code(\Sigma)), w(Code(\Lambda))\geq \sup\{w(A) : A\in \Gamma\}$ \ \ \ \ \ \ \ (*).
\end{center}
 Then, using the version of comparison theorem, \rthm{comparison}, for pairs that are $\Gamma$-fullness preserving, we get a contradiction as follows. Let $\R\in I(\P, \Sigma)\cap I(\Q, \Lambda)$ be such that $\Sigma_\R=\Lambda_\R$. Let $\Psi=\Sigma_\R=\Lambda_\R$. Then
 \begin{center}
 $Code(\Psi)\in L(A, \mathbb{R})\cap L(B, \mathbb{R})$
 \end{center}
 and hence, $Code(\Psi)\in \Gamma$. Let then $i: \P\rightarrow \R$ and $j:\Q\rightarrow \R$ be the iteration embeddings. It follows from branch condensation, or rather one of its consequences, the \textit{pullback consistency}, of $\Sigma$ and $\Lambda$, that $\Sigma$ is $i$-pullback of $\Psi$ and $\Lambda$ is the $j$-pullback of $\Psi$. Hence, $\Sigma, \Lambda\in L(\Gamma, \mathbb{R})$, which is a contradicts (*). 
 \end{proof}

The same proof actually gives a stronger result which can be used to prove \rthm{lst from wlw}

 \begin{theorem}\label{div models revised}
Suppose  $L(A,\mathbb{R})$ and $ L(B,\mathbb{R})$ are divergent models of $ AD^{+}+V=L(\powerset(\mathbb{R}))$. Then there is an $LST$ hod premouse. 
\end{theorem}

The problem of evaluating the exact consistency strength of the existence of divergent models of $AD^+$ is  Problem 17 of \cite{OpenProblems}. We have already remarked that Woodin showed that a Woodin limit of Woodin cardinals is an upper bound (see \cite{EA}). The author, in unpublished work, showed that the existence of divergent models of $AD^{+}+V=L(\powerset(\mathbb{R}))+MC+\theta_0=\theta$ gives a model with a Woodin limit of Woodin cardinals. Woodin's construction can be used to obtain divergent models of $AD^{+}+V=L(\powerset(\mathbb{R}))+MC+\theta_0=\theta$ from a Woodin limit of Woodin cardinals. Thus, we get the following theorem.

\begin{theorem}[S.-Woodin] The following theories are equiconsistent.
\begin{enumerate}
\item ZF+there are divergent models of $AD^{+}+V=L(\powerset(\mathbb{R}))+MC$.
\item ZFC+there is a Woodin limit of Woodin cardinals.
\end{enumerate}
\end{theorem}

We conjecture that the exact consistency strength of divergent models of $AD^{+}$ is that of a Woodin limit of Woodin cardinals. It follows from \rthm{div models revised} that to prove the aforementioned conjecture, it is enough to prove \rcon{lst conjecture}.

\section{Concluding remarks}\label{speculations}

In order to solve the inner model problem via the methods described in this paper, either using the above idea or ideas from \rsec{partial results on imp}, we need to at least know that the Solovay hierarchy consistencywise ``catches up" with the large cardinal hierarchy. This is the content of \rprob{main problem}. A very plausible approach to \rprob{main problem} is via proving its revised form, \rcon{dimt conjecture}, namely that for each large cardinal axiom $\phi$, $S_\phi$ is consistent. The intuition behind \rcon{dimt conjecture} has to do with the $\H$ analysis outlined in the previous section. This analysis will eventually show that $\H$ of models of $AD^+$ is a hod mouse and hence, it has significant large cardinals in it. 

To establish the consistency of $S_\phi$ one needs to develop a general theory of hod mice without any minimality conditions. Woven into this theory will be a comparison lemma which will then be used to settle \rcon{the hod conjecture} and \rcon{the generation of full pointclasses}. Solving \rcon{the capturing of hod pairs} will probably be harder but nevertheless a good theory of hod mice should settle it. The three conjectures will then imply MSC.

If for some large cardinal axiom $\phi$, $S_\phi$ is inconsistent then everything we have said so far gets into a better position from one point of view and in a devastating position from another. It will be good because it would imply that to solve MSC one only needs to go as far as $\phi$. It will devastate our picture because it will defeat the philosophy of this paper in a way that at the moment we do not see a way around. The philosophy of the paper has been that the Solovay hierarchy covers all levels of the consistency strength hierarchy and to solve the inner model problem one essentially only needs to understand the levels of the Solovay hierarchy. If for some $\phi$, $S_\phi$ is inconsistent yet that doesn't translate into an inconsistency in the large cardinal hierarchy then our philosophy is false and at the moment, our mathematical imagination seems to be too restrictive to perceive such a set theoretic reality.
\bibliographystyle{plain}
\bibliography{BSL}

\end{document}